\documentclass[12pt]{article}

 \usepackage{amsmath,amssymb,amscd,amsthm,esint}
 
\usepackage{graphics,amsmath,amssymb,amsthm,mathrsfs}

\usepackage{graphics,amsmath,amssymb,amsthm,mathrsfs,amsfonts,accents}

\usepackage{sidecap}
\usepackage{float}
\usepackage{extarrows}
\usepackage{booktabs}
\usepackage{verbatim}
\usepackage{hyperref}
\usepackage[usenames,dvipsnames]{xcolor}

\newtheorem{thm}{Theorem}[section]
\newtheorem{lemma}[thm]{Lemma}
\newtheorem{cor}[thm]{Corollary}

\theoremstyle{definition}
\newtheorem{remark}[thm]{Remark}

\def\XXint#1#2#3{{\setbox0=\hbox{$#1{#2#3}{\int}$}
         \vcenter{\hbox{$#2#3$}}\kern-.5\wd0}}

\def\R{\mathbb{R}}
\def\C{\mathbb{C}}
\def\e{\varepsilon}

\def\Rd{\mathbb{R}^{d}}
\def\Rdp{\mathbb{R}^d_+}
\def\Cd{\mathbb{C}^d}
\def\Cdd{\mathbb{C}^{d\times d}}

\def\hW{\accentset{\circ}{W}}
\def\wu{\widetilde{u}}
\def\wp{\widetilde{p}}

\def\loc{{\rm loc}}

\def\H{\mathbb{H}}

\numberwithin{equation}{section}

\begin{document}

\title{Resolvent Estimates   for the Stokes Operator \\ in Bounded and Exterior $C^1$ Domains}

\author{Jun Geng  \thanks{Supported in part by NNSF grant 12371096.}
\qquad
Zhongwei Shen \thanks{Supported in part by NSF grant DMS-2153585.}}
\date{}

\maketitle

\begin{abstract}

We establish  resolvent estimates in $L^q$ spaces for the Stokes operator in a bounded   $C^1$ domain  $\Omega$  in $\Rd$.
As a corollary, it follows  that the Stokes operator generates a bounded analytic semigroup in  $L^q(\Omega; \C^d)$
for any $1< q< \infty$ and $d\ge 2$. The case of an exterior $C^1$ domain is also studied. 

\medskip


\medskip

\noindent {\it MR (2010) Subject Classification}: 35Q30.

\end{abstract}

\section{Introduction}

In this paper we study  the resolvent  problem for the  Stokes operator with the Dirichlet condition,
\begin{equation}\label{eq-0}
\left\{
\aligned
-\Delta  u +\nabla p +\lambda u & = F  & \quad & \text{ in } \Omega, \\
\text{\rm div}(u) & =0 & \quad  & \text{ in } \Omega,\\
u&=0 & \quad & \text{ on } \partial\Omega,
\endaligned
\right.
\end{equation}
where $\lambda\in \Sigma_\theta $ is a  parameter and 
\begin{equation}\label{Sigma}
\Sigma_\theta
=\left\{
z\in \mathbb{C}\setminus \{ 0\}: \ |\text{arg} (z)|< \pi -\theta
\right\}
\end{equation}
for $\theta \in (0, \pi/2)$.
The following two theorems are the main results of the paper.
The first one covers the case of bounded domains with $C^1$ boundaries,
while the second treats the case of exterior $C^1$ domains.

\begin{thm}\label{main-1}
Let $\Omega$ be a bounded $C^1$ domain in $\Rd$, $d\ge 2$.
Let $1<q<\infty$ and $\lambda \in \Sigma_\theta$.
Then for any $F\in L^q(\Omega; \Cd)$, the Dirichlet problem \eqref{eq-0} has a unique solution $(u, p)$  in 
$W_0^{1, q}(\Omega; \Cd) \times L^q(\Omega; \C)$ with $\int_\Omega p=0$.
Moreover, the solution satisfies the estimate,
\begin{equation}\label{est-0}
(|\lambda| + 1)^{1/2} \| \nabla u \|_{L^q(\Omega)}
+ (|\lambda|+1) \| u \|_{L^q(\Omega)}
\le C  \| F \|_{L^q(\Omega)}, 
\end{equation}
where $C$ depends only on $d$, $q$, $\theta$ and $\Omega$.
\end{thm}

\begin{thm}\label{main-2}
Let $\Omega$ be an exterior domain with $C^1$ boundary  in $\Rd$, $d\ge 2$.
Let $1<q<\infty$ and $\lambda\in \Sigma_\theta$ with $|\lambda|\ge \delta>0$.
Then for any $F\in L^q(\Omega; \Cd)$, the Dirichlet problem \eqref{eq-0} has a unique solution $(u, p)$  in 
$W_0^{1, q}(\Omega; \Cd) \times L^q_{\loc}(\overline{\Omega}; \C) $.
Moreover, the solution satisfies the estimate,
\begin{equation}\label{est-1}
|\lambda|^{1/2} \| \nabla u \|_{L^q(\Omega)}
+ |\lambda| \| u \|_{L^q(\Omega)}
\le C  \| F \|_{L^q(\Omega)}, 
\end{equation}
where $C$ depends only on $d$, $q$, $\theta$, $\delta$ and $\Omega$.
Furthermore, if $d\ge 3$, the estimate, 
\begin{equation}\label{est-2}
 |\lambda| \| u \|_{L^q(\Omega)}\le C \| F \|_{L^q(\Omega)},
 \end{equation} 
holds with $C$ independent of $\delta$.
\end{thm}

Resolvent estimates for the Stokes operator play an essential role in the functional analytic approach of Fujita and Kato \cite {Kato-1964} 
 to the nonlinear  Navier-Stokes equations. 
 The resolvent estimate \eqref{est-2}   in  domains with smooth boundaries 
 has been studied extensively since 1980's.
 Under the assumption that $\Omega$ is a bounded or exterior domain  with $C^{1, 1}$ boundary, the estimate \eqref{est-2}
 holds  for any $1< q<\infty$ \cite{Solo-1977, Giga-1981, Sohr-1987,  Sohr-1994}.
 We refer the reader to \cite{Sohr-1994} for a review as well as a comprehensive list of references in the case of smooth domains.
 The recent work in this area focuses on domains with nonsmooth boundaries. 
 If $\Omega$ is merely a bounded Lipschitz domain,
 it was proved by one of the present  authors \cite{Shen-2012} that the resolvent estimate \eqref{est-2}  holds  if  $d\ge 3$ and 
 \begin{equation}\label{Lip}
 \left| \frac{1}{q}-\frac12 \right| < \frac{1}{2d} +\e,
 \end{equation}
 where $\e>0$ depends on $\Omega$.
 In particular, 
 in  the case $d=3$, this shows that  the  estimate  \eqref{est-2} holds   for $(3/2)-\e < q< 3+\e$
 and gives an affirmative answer to a conjecture of M. Taylor \cite{Taylor-2000}.
 For a two-dimensional bounded Lipschitz domain, F. Gabel and P. Tolksdorf \cite{Tolksdorf-2022} were able to
 establish  the  resolvent estimate 
 \eqref{est-2}  for 
 $(4/3)-\e< q< 4+\e$.
   It is not known whether the range in \eqref{Lip} is sharp for Lipschitz domains. 
 In \cite{Deuring-2001} P. Deuring  constructed an interesting example of an unbounded Lipschitz domain  for which the resolvent 
 estimate fails for large $q$.
 For   related work on the Stokes and
  Navier-Stokes equations in Lipschitz or $C^1$ domains,
 we refer to the reader to \cite{FKV-1988, Brown-S-1995, Mitrea-2004, Mitrea-2008,  Mitrea-2009, Mitrea-2012, Kilty-2015, Tolksdorf-2018, Tolksdorf-2020} 
 
 The main contribution of this paper lies  in the smoothness assumption for  the domain $\Omega$.
 We are able to establish the resolvent estimates for the full range $1< q< \infty$ under the assumption 
 that $\partial \Omega$ is $C^1$. In view of  the example by P. Deuring \cite{Deuring-2001}, this assumption is more or less
 optimal. As we mentioned earlier, the  full range is  known previously  for $C^{1,1}$ domains \cite{Sohr-1994}. 
 A recent result of D. Breit \cite{Breit-2022} implies  the resolvent estimates for a three-dimensional Lipschitz domain satisfying certain 
 Besov-type conditions, which are 
 weaker than $C^{1, 1}$ and somewhat close to $C^{1, \alpha}$ for certain  $\alpha>0$. 
 Note  that in the case of smooth domains, in addition to  the $L^q$ estimates for $u$ and $\nabla u$ in  \eqref{est-0} and \eqref{est-1},
  one also obtains  an estimate for $\nabla^2 u$,
 \begin{equation}\label{est-3}
 \| \nabla ^2 u \|_{L^q(\Omega)} \le C \| F \|_{L^q(\Omega)},
 \end{equation}
 for $1<q< \infty$,  if $\Omega$ is bounded (some restrictions on $q$ are needed if $\Omega$ is an exterior domain; see \cite{Sohr-1994}).
 However,  such $W^{2, q}$ estimates fail in $C^1$ domains, even for the Laplace operator.
 
 Let $C^\infty_{0, \sigma}(\Omega)=\left\{ u \in C_0^\infty(\Omega; \Cd): \text{\rm div}(u)=0 \right\}$
and
\begin{equation}\label{div}
L^q_\sigma (\Omega)
= \text{\rm the closure of } C_{0, \sigma}^\infty (\Omega) \text{ in } L^q (\Omega;  \Cd).
\end{equation}
For $1< q<\infty$, we define the Stokes operator $A_q $ in $L^q_\sigma (\Omega)$ by
\begin{equation}\label{S-op}
A_q (u) = -\Delta u +\nabla p,
\end{equation}
with the domain
\begin{equation}\label{Domain}
\aligned 
\mathcal{D}(A_q)
=\Big\{ u\in W^{1, q}_0(\Omega; \Cd): & \  \text{\rm div}(u)=0 \text{ in } \Omega \text{ and } \\
& -\Delta u +\nabla p \in L^q_\sigma (\Omega) \text{ for some } p \in L_{\loc}^q(\overline{\Omega}; \C) \Big\}.
\endaligned
\end{equation}
It  follows from Theorems \ref{main-1} and \ref{main-2}  that for $\lambda\in \Sigma_\theta$ and $1< q< \infty$, the inverse operator 
$(\lambda + A_q)^{-1}$ exists as a bounded operator on $L^q_\sigma (\Omega)$. Moreover, 
the estimate, 
\begin{equation}\label{re-est} 
\| (\lambda + A_q)^{-1} F  \|_{L^q(\Omega)}  \le  C |\lambda|^{-1} \| F \|_{L^q(\Omega)},
\end{equation}
holds, where $C$ depends only on $d$, $q$, $\theta$ and $\Omega$, if $\Omega$ is a bounded $C^1$ domain  in $\R^d$, $d\ge 2$ or 
an exterior $C^1$ domain in $\R^d$, $d\ge 3$.
As a corollary, we obtain the following.

\begin{cor}\label{main-3}
Let $\Omega$ be a bounded $C^1$ domain in $\Rd$, $d\ge 2$  or an exterior $C^1$ domain in $\R^d$, $d\ge 3$.
Then the Stokes operator $-A_q$ generates a uniformly  bounded analytic semigroup $\{ e^{-tA_q} \}_{t\ge 0} $ in 
$L^q_\sigma (\Omega)$ for $1< q< \infty$.
\end{cor}

The uniform boundedness of the semigroup in the case of two-dimensional exterior $C^1$ domains is left open by Corollary \ref{main-3}.
We note that the uniform boundedness for the two-dimensional exterior $C^2$ domains was established in \cite{Borchers-two} by 
using the method of layer potentials for $\lambda$ near $0$.

 We now describe our approach to Theorems \ref{main-1} and \ref{main-2}, which is based on a perturbation argument of R. Farwig and
 H. Sohr  \cite{Sohr-1994}.
 The basic idea is to work out first  the cases of the whole space $\R^d$ and the half-space $\Rdp$.
 One then uses a perturbation argument to treat the case of a region above a graph,
 $$
 \H_\psi =\left\{ (x^\prime, x_d) \in \R^d: \ x^\prime \in \R^{d-1} \text{  and } x_d > \psi (x^\prime) \right\},
 $$
 where $\psi: \R^{d-1} \to \R$.
 Finally, a localization procedure, together with some compactness argument,  is performed  to handle  the cases of bounded or exterior domains.
 To establish the resolvent estimates for $C^1$ domains, the key step  is to carry out the perturbation argument 
 under the assumption that $\psi: \R^{d-1} \to \R$ is Lipschitz continuous
  and to show that the  error terms are  bounded by  the Lipschitz norm $\|\nabla^\prime \psi \|_\infty$,
  where $\nabla^\prime$ denotes the gradient with respect to $x^\prime =(x_1, \dots, x_{d-1})$.
 
 To this end, we consider   a more general Stokes resolvent problem,
 \begin{equation}\label{eq-g}
 \left\{
 \aligned
 -\Delta u +\nabla p +\lambda u & = F +\text{\rm div} (f),\\
 \text{ \rm div} (u) & =g,
 \endaligned
 \right.
 \end{equation}
 in $\H_\psi$ with the boundary condition $u =0$ on $\partial\H_\psi$,  
 where $F\in L^q(\H_\psi; \Cd)$ and $f\in L^q(\H_\psi; \Cdd)$.
 We introduce two Banach spaces,
 \begin{equation}\label{XY}
 X_\psi^q = W^{1, q}_0(\H_\psi ; \Cd) \times A_\psi^q \quad
 \text{ and } \quad
 Y_\psi^q = W^{-1, q}(\H_\psi; \Cd)\times B_\psi^q,
 \end{equation}
 where $A_\psi^q$ and $B_\psi^q$ are two spaces  defined by \eqref{AB}. 
 In  comparison with the spaces used in \cite{Sohr-1994} for $C^{1, 1}$ domains,
 we point out that since we work with $C^1$ domains, no $W^{2, q}$ spaces can be used.
Note that  the scaling-invariant property  of the Lipschitz norm $\|\nabla^\prime \psi\|_\infty$ allows  us to 
 fix $\lambda\in \Sigma_\theta$ with $|\lambda|=1$. Consider the linear operator
 \begin{equation}\label{S00}
 S_\psi^\lambda (u, p)  = (-\Delta u +\nabla p +\lambda u, \text{\rm div} (u) ).
 \end{equation}
 We are able to show that $S_\psi^\lambda: X_\psi^q \to Y_\psi^q $ is a bijection and that 
 \begin{equation}\label{S-0}
 \| (S_\psi^\lambda)^{-1}  \|_{Y_\psi^q \to X_\psi^q} \le C(d, q, \theta)
 \end{equation}
 for $1< q<\infty$, provided that  $\|\nabla^\prime \psi  \|_\infty\le  c_0$ and $c_0=c_0(d, q, \theta)>0$  is sufficiently small.
 See Theorem \ref{G-thm-1}.
 To prove \eqref{S-0}, one first considers the special case $\psi=0$;  i.e., $\H_\psi = \Rdp$.
 The general case follows from the facts that 
 \begin{equation}\label{per-0}
 S_\psi^\lambda (u, p) = S_0^\lambda (\wu, \wp) \circ \Psi + R(\wu, \wp)\circ \Psi, 
 \end{equation}
 and that the operator norm of the second term in the right-hand side of \eqref{per-0}  is bounded by  $ C \|\nabla^\prime \psi \|_\infty$
 if $\| \nabla^\prime \psi  \|_\infty \le 1$.
 As a by-product, we also obtain the resolvent estimate \eqref{est-1} in the case $\Omega=\H_\psi$ if
 $\|\nabla^\prime  \psi \|_\infty \le  c_0 (d, q, \theta)$.
 See Theorem \ref{thm-G}.
 
 The paper is organized as follows.
We start with the case of the whole space $\R^d$  in  Section \ref{section-R}.
The case $\Omega=\Rdp$ is studied  in Section \ref{section-H}.
In Section \ref{section-G} we carry out the perturbation argument described above for the region above a Lipschitz graph.
In Section \ref{section-B} we consider the case of bounded $C^1$ domains and give the proof of Theorem \ref{main-1}.
The case of exterior $C^1$  domains is studied in Section \ref{section-E}, where Theorem \ref{main-2}  is proved. 
Finally, we  prove some useful uniqueness and  regularity results for exterior  $C^1$ domains in the Appendix. 
 
 We end this section with a few notations  that will be  used throughout the paper.
 Let $\Omega$ be a (bounded or unbounded) domain in $\R^d$. 
  By $u \in L^q_{\loc} (\overline{\Omega}; \C^m)$  we mean $u \in L^q(B\cap \Omega; \C^m)$ for any ball $B$ in $\Rd$.
 For $1<q<\infty$, let
 \begin{equation}\label{W}
 W^{1, q}(\Omega; \C^m) 
 =\left\{  u \in L^q(\Omega; \C^m): \ \nabla u \in L^q(\Omega; \C^{d \times m}) \right\}
 \end{equation}
 be the usual Sobolev space  in $\Omega$ for functions with values in $\C^m$.
 By $W^{1, q}_0(\Omega; \C^m)$ we denote the closure of $C_0^\infty(\Omega; \C^m)$ in $W^{1, q}(\Omega; \C^m)$.
 We use $W^{-1, q}(\Omega; \C^m)$ to denote the dual of $W^{1, q^\prime}_0(\Omega; \C^m)$ and
 $W^{-1, q}_0(\Omega; \C^m)$ the dual of $W^{1, q^\prime}(\Omega; \C^m)$, where
 $q^\prime =\frac{q}{q-1}$.
 For $1< q<\infty$, we  let
\begin{equation}\label{hW}
\hW^{1, q} (\Omega; \C^m)=\left\{ u \in L^q_{\loc} (\overline{\Omega} ; \C^m): \ \nabla u \in L^q(\Omega; \C^{d \times m} ) \right\}
\end{equation}
denote the homogeneous $W^{1, q}$ space  with the norm $\| \nabla u\|_{L^q(\Omega)}$.
As usual, we identify two functions  in $ \hW^{1, q} (\Omega; \C^m)$ if they differ by a constant.
Let $\hW^{-1, q}(\Omega; \C^m)$ be the dual of $\hW^{1, q^\prime}(\Omega; \C^m )$.
Elements $\Lambda$  in $\hW^{-1, q}(\Omega; \C^m)$ may be represented by $\text{\rm div} (f)$,
where $f=(f_{jk}) \in L^q(\Omega; \C^{d \times m} )$,  in the sense that
$$
\Lambda (u) = - \int_\Omega \partial_j u_k \cdot f_{jk}
$$
for any $u=(u_1, \dots, u_m)  \in \hW^{1, q^\prime }(\Omega; \C^m)$,
where $\partial_j =\partial /\partial x_j$,  the index  $j$ is summed from $1$ to $d$ and $k$ from $1$ to $m$.

\medskip

\noindent{\bf Acknowledgement.}
The authors thank the anonymous referees for their helpful comments that improved the quality of the manuscript.

%%%%%%%%%%%%%%%%%%%%%%%%%%%%%%%%%%%%%%%%%%%%%%%%%%%%%%%%

\section{The whole space}\label{section-R}

In this section we study  the resolvent problem for the Stokes equations in $\Rd$, $d\ge 2$.
The results in Theorem \ref{R-thm-1}  are more or less standard.
Since  the Stokes equations  are considered with a  more general  data set, 
we provide a proof for the reader's convenience.

\begin{thm}\label{R-thm-1} 
Let  $1<q<\infty$ and $\lambda\in \Sigma_\theta $.
For any $F\in L^q(\Rd; \Cd)$,  $f\in L^q(\Rd; \Cdd)$, and $g\in L^q(\Rd; \C)\cap \hW^{-1, q}(\Rd; \C)$,
there exists  a unique $u\in W^{1, q}(\Rd; \Cd) $ such that 
\begin{equation}\label{R-eq}
\left\{
\aligned
-\Delta u +\nabla p +\lambda u & = F +\text{\rm div} (f) ,\\
\text{\rm div} (u) & =g 
\endaligned
\right.
\end{equation}
hold in $\Rd$ for some $p\in L^1_{\loc}(\Rd; \C)$ in the sense of distributions.
Moreover, the solution  satisfies the estimate,
\begin{equation}\label{R-est-1}
\left\{
\aligned
|\lambda |^{1/2} \| \nabla u \|_{L^q (\Rd)} 
 & \le C \left\{ \| F \|_{L^q(\Rd)} + |\lambda|^{1/2} \| f \|_{L^q(\Rd)} + |\lambda|^{1/2} \| g \|_{L^q(\Rd)} \right\}, \\
|\lambda|
\| u \|_{L^q(\Rd)}
& \le C \left\{   \| F \|_{L^q(\Rd)} + |\lambda|^{1/2} \| f \|_{L^q(\Rd)} +  |\lambda | \| g \|_{\hW^{-1, q} (\Rd)} \right\}, 
\endaligned
\right.
\end{equation}
and $p\in L^q(\Rd; \C) + \hW^{1, q} (\Rd; \C)$, 
where $C$ depends on $d$, $q$ and $\theta$.
\end{thm}

\begin{proof}

Step 1. We establish  the existence of the solution and the estimates in \eqref{R-est-1}.

By  rescaling we may assume $|\lambda|=1$.
% We begin with  the uniqueness. 
%Suppose that   $u\in W^{1, q}(\Rd; \Cd)$, $p\in  L^q(\Rd; \C) + \hW^{1, q}(\Rd; \C)$, and $(u, p)$ satisfy  \eqref{R-eq} with 
%$F=0$, $f=0$ and $g=0$.
%By applying the divergence operator to the equation 
%$-\Delta u +\nabla p+\lambda u=0$, we obtain $\Delta p=0$ in $\Rd$.
%By the classical  interior estimates for harmonic functions,
 %we have
 %$$
 %\nabla p(x_0) |\le \frac{C}{R}  \left(\fint_{B(x_0, R)} |p -\alpha |^q \right)^{1/q}
 %$$
 %for any $x_0\in \Rd$, $R>0$ and $ \alpha \in \C$, where $C$ depends on $d$ and $q$.
 %Suppose that $p=p_1 +p_2$, where $p_1\in L^q(\Rd; \C)$ and $p_2\in \hW^{1, q}(\Rd; \C)$.
 %t follows that
 %\begin{equation}\label{uni-1}
 %\aligned
 %|\nabla p (x_0)|
% &  \le \frac{C}{R}  \left(\fint_{B(x_0, R)} |p_1  |^q \right)^{1/q}
% + \frac{C}{R}  \left(\fint_{B(x_0, R)} |p_2-\alpha   |^q \right)^{1/q}\\
 % &  \le \frac{C}{R}  \left(\fint_{B(x_0, R)} |p_1  |^q \right)^{1/q}
 % +   C  \left(\fint_{B(x_0, R)} |\nabla p_2  |^q \right)^{1/q},
 % \endaligned
 % \end{equation}
 %where we have chosen $\alpha = \fint_{B(x_0, R)} p_2$ and used Poincar\'e's inequality.
% By letting $R\to \infty$ in \eqref{uni-1}, we deduce that $\nabla p=0$ in $\Rd$ and thus  $p$ is constant. 
%It follows that $-\Delta u +\lambda u=0$ in $\Rd$.
%Since $\lambda \in \Sigma_\theta$,  this  implies that
%\begin{equation}\label{uni-2}
%|u(x_0)|\le C_\theta \left(\fint_{B(x_0, R)} |u|^q \right)^{1/q}
%\end{equation}
%for any $x_0\in \Rd$ and $R>1$, where $C_\theta$ depends on $d$, $q$ and $\theta$.
%By letting $R \to \infty$ in \eqref{uni-2} we obtain $u=0$ in $\Rd$.
By linearity, it suffices to consider two cases: (I) $g=0$; (II) $f=0$ and $F=0$.

Case I. Assume $g=0$.
  Let  $\mathcal{F}$ denote the Fourier transform defined by 
$$
\mathcal{F}(h) (\xi )=\int_{\Rd} e^{-i x\cdot \xi }  h(x)\, dx,
$$
where $i=\sqrt{-1}$ and $\xi \in \Rd$.
Let $u=(u_1, u_2, \dots, u_d)$, $F=(F_1, F_2, \dots, F_d) $  and $f=(f_{jk})$. 
By applying $\mathcal{F}$  to \eqref{R-eq} with $g=0$, we obtain 
\begin{equation}\label{R-2}
\left\{
\aligned
(|\xi|^2 +\lambda) \mathcal{F} (u_j) + i \xi_j  \mathcal{F}(p)  &  = \mathcal{F}(F_j) + i \xi_\ell \mathcal{F} (f_{\ell j} )  & \quad & \text{ in } \Rd,\\
\xi_\ell \mathcal{F} (u_\ell) & =0 & \quad & \text{ in } \Rd,
\endaligned
\right.
\end{equation}
where the repeated index $\ell$ is summed from $1$ to $d$.
A solution of \eqref{R-2} is given by 
\begin{equation}\label{R-3}
\left\{
\aligned
\mathcal{F}(u_j) & = (\lambda + |\xi|^2)^{-1} \left( \delta_{jk} -\frac{\xi_j \xi_k}{ |\xi|^2} \right) 
\left( \mathcal{F} (F_k) + i \xi_\ell  \mathcal{F}(f_{\ell k}) \right), \\
\mathcal{F} (p) & =  \frac{-i \xi_k}{|\xi|^2}
\left( \mathcal{F} (F_k) + i \xi_\ell  \mathcal{F}(f_{\ell k}) \right),
\endaligned
\right.
\end{equation}
where the repeated indices $k, \ell $ are summed from $1$ to $d$.
Since $\lambda\in \Sigma_\theta$ and $|\lambda|=1$, we have $  |\lambda + |\xi|^2 |  \approx 1+ |\xi|^2$.
Thus,  by the Mikhlin  multiplier theorem, there exist $u \in W^{1, q}(\Rd; \Cd)$ and $p \in L^q(\Rd; \C)
+ \hW^{1, q}(\Rd; \C)$, 
satisfying \eqref{R-eq} and
\begin{equation} 
 \|\nabla u \|_{L^q(\Rd) } + \| u \|_{L^q(\Rd)}
\le C\left\{  \| f \|_{L^q(\Rd) } + \| F \|_{L^q(\Rd)} \right\}, 
\end{equation}
for $1<q<\infty$, where $C$ depends on $d$, $q$ and $\theta$.

Case II. Assume that $F=0$ and $f=0$.
Since $g\in L^q(\Rd; \C) \cap \hW^{-1, q}(\Rd; \C)$,
there exists $G\in \hW^{1, q}(\Rd; \C)$ such that $\nabla G \in W^{1, q}(\Rd; \Cd)$,  $\Delta G =g $ in $\Rd$, 
$$
\| \nabla G \|_{L^q(\Rd)} \le C \| g \|_{\hW^{-1, q}(\Rd)}
\quad \text{ and } \quad
\| \nabla^2 G \|_{L^q (\Rd)} \le C \| g \|_{L^q(\Rd)}.
$$
Let $u=\nabla G$ and $p=g-\lambda G$. Then
$u \in W^{1, q} (\Rd; \Cd)$, $p\in L^q(\Rd; \C) + \hW^{1, q}(\Rd; \C)$, and $(u, p)$ satisfies \eqref{R-eq}
with $F=0$ and $f=0$.
Moreover,
$$
\| \nabla u \|_{L^q(\Rd)} \le C \| g \|_{L^q(\Rd)}
\quad
\text{ and } 
\quad
\| u \|_{L^q(\Rd)} \le C \| g \|_{\hW^{-1, q}(\Rd)}.
$$

Step 2. We establish the uniqueness of the solution.

 Let  $u\in W^{1, q} (\Rd; \Cd)$ be a solution of \eqref{R-eq} in $\Rd$ with 
$F=0$, $f=0$ and $g=0$.
It follows that  for any $w\in C_{0, \sigma}^\infty (\Rd)$, 
\begin{equation}\label{R-5}
\int_{\Rd} \nabla u \cdot \nabla w + \lambda \int_{\Rd} u \cdot w=0,
\end{equation}
where $C^\infty_{0, \sigma} (\Rd)= \{ w\in C_0^\infty(\Rd; \Cd): \text{\rm div} (w)=0 \text{ in } \Rd \}$. Since $u\in W^{1, q}(\Rd; \Cd)$, 
by a density argument, we deduce  that \eqref{R-5} holds for any $w\in W^{1,q^\prime}(\Rd; \Cd)$ with div$(w)=0$ in $\Rd$.
Let $w$ be a solution in $W^{1, q^\prime} (\Rd; \Cd)$  of the Stokes equations,
\begin{equation}\label{R-6}
\left\{
\aligned
-\Delta w +\nabla \phi + \lambda w & = |u|^{q-2} \overline{u},\\
\text{\rm div} (w) & =0 
\endaligned
\right.
\end{equation}
in $\Rd$, where $\overline{u}$ denotes the complex conjugate of $u$.
Since $|u|^{q-2} \overline{u} \in L^{q^\prime} (\Rd; \Cd)$, such solution exists in $W^{1, q^\prime}(\Rd; \Cd)$
by Step 1. Again by a density argument, we may deduce from \eqref{R-6} that
\begin{equation}\label{R-7} 
\int_{\Rd} \nabla w \cdot \nabla u + \lambda \int_{\Rd} w \cdot u 
=\int_{\Rd} |u|^q.
\end{equation}
In view of \eqref{R-5} and \eqref{R-7}, we obtain $\int_{\Rd} |u|^q=0$ and thus $u=0$ in $\Rd$.
\end{proof}

\begin{remark}\label{remark-R}
Let $F$, $f$, $g$, $(u, p)$ be the same as in Theorem \ref{R-thm-1}.
Let $F=(F_1,F_2,  \dots, F_d)$ and $f=(f_{jk})$. The $k$ component of  $\text{\rm div} (f) $ is given by
$\sum_j  \partial_j f_{jk} $, where $\partial_j$ denotes  $\partial/\partial x_j$.
Let $x=(x^\prime, x_d)$, where $x^\prime \in \mathbb{R}^{d-1}$.
Suppose that
\begin{equation}\label{ref}
\left\{
\aligned
& \text{ 
$F_j$ is even in $x_d$ for $1\le j \le d-1$ and $F_d$ is odd, }\\
& \text{
$g$ is even in $x_d$, }\\
&\text{ 
 $ f_{jk} $ is even in $x_d$  for $1\le j, k \le d-1$,}\\
  & \text{ $f_{dd} $ is even in $x_d$, }\\
& \text{
$f_{jd}$ and $f_{dj}$ are odd in $x_d$ for $1\le j \le d-1$.}\\
\endaligned
\right.
\end{equation}
Define
$$
\left\{
\aligned
v (x^\prime, x_d)  & = (u_1 (x^\prime, -x_d), \dots, u_{d-1} (x^\prime, -x_d), -u_d (x^\prime, -x_d)), \\
\phi  (x^\prime, x_d) & = p(x^\prime, -x_d).
\endaligned
\right.
$$
Then $(v, \phi  )$ is a solution of \eqref{R-eq} with the same data $F, f$ and $g$.
By the uniqueness in Theorem \ref{R-thm-1}, it follows that $u=v$ in $\Rd$.
In particular, this implies that $u_d (x^\prime, 0) =0$ for $x^\prime\in \mathbb{R}^{d-1}$.
\end{remark}

\begin{remark}\label{re-R-p}
Assume $\lambda\in \Sigma_\theta$ and $|\lambda|=1$.
Let $(u, p)$ be the solution of \eqref{R-eq}, given by Theorem \ref{R-thm-1}.
An inspection of the proof of Theorem \ref{R-thm-1}  shows that $p=p_1 +p_2$, where
$p_1\in L^q(\Rd; \C)$,  $p_2\in \hW^{1, q}(\Rd; \C)$, and
$$
\| p_1 \|_{L^q(\Rd)} + \|\nabla p_2\|_{L^q(\Rd)}
\le C \left\{ \| F \|_{L^q(\Rd)} + \| f \|_{L^q(\Rd)} + \| g \|_{L^q(\Rd)}
+ \| g \|_{\hW^{-1, q} (\Rd)} \right\}.
$$
The constant $C$ depends only on $d$, $q$ and $\theta$.
\end{remark}

\begin{remark}\label{re-R-u}
Let $1<q_1<q_2<\infty$ and $\lambda\in \Sigma_\theta $.
Suppose that  $F\in L^{q_j} (\Rd; \Cd)$, $f\in L^{q_j}  (\Rd; \Cdd)$ and $g\in L^{q_j} (\Rd; \C) \cap \hW^{-1, q_j}(\Rd; \C)$
for $j=1, 2$.
Let $(u^j, p^j)$ be the unique solution of \eqref{R-eq} in $W^{1, q_j}(\Rd; \Cd)\times (L^{q_j}(\Rd; \C) + \hW^{1, q_j}(\Rd; \C))$,
given by Theorem \ref{R-thm-1}.
Then $(u^1, p^1)=(u^2, p^2)$.
% To see this, we let $(v, \phi)=(u^1-u^2, p^1-p^2)$ and apply the argument for the uniqueness in the proof of Theorem \ref{R-thm-1}
%to $(v, \phi)$. 
This  follows from the observation  that the solutions constructed in the proof do not depend on $q$. 
\end{remark}

%%%%%%%%%%%%%%%%%%

\section{A half-space}\label{section-H}

In this section we consider the resolvent problem for the Stokes equations in the half-space $\Rdp$.
Recall that $\hW^{1, q}(\Rdp; \C)$ is  the homogeneous $W^{1, q}$ space in $\Rdp$ defined by \eqref{hW}, and 
$\hW^{-1, q}(\Rdp; \C)$ denotes  the dual of $\hW^{1, q^\prime} (\Rdp; \C)$.

\begin{thm}\label{H-thm-1}
Let $1< q< \infty$ and $\lambda\in \Sigma_\theta$.
Let $F\in L^q(\Rdp; \Cd)$, $f\in L^q(\Rdp; \Cdd)$, and $g\in L^q(\Rdp; \C)\cap \hW^{-1, q}(\Rdp; \C)$.
Then there exists    a unique  $u\in W_0^{1, q}(\Rdp; \Cd)$ such that 
\begin{equation}\label{H-eq}
\left\{
\aligned
-\Delta u+\nabla p +\lambda u & = F + \text{\rm div}(f), \\
\text{\rm div} (u) & =g 
\endaligned
\right.
\end{equation}
hold in $\Rdp$ for some $p\in L^1_{\loc} (\Rdp; \C)$  in the sense of distributions.
Moreover,  the  solution $u$  satisfies the estimate, 
\begin{equation}\label{H-est-1}
\aligned
& |\lambda |^{1/2} \| \nabla u \|_{L^q (\Rdp)} 
+|\lambda|
\| u \|_{L^q(\Rdp)}\\
 & \le C \left\{ \| F \|_{L^q(\Rdp)} + |\lambda|^{1/2} \| f \|_{L^q(\Rdp)} + |\lambda|^{1/2} \| g \|_{L^q(\Rdp)} 
 +  |\lambda | \| g \|_{\hW^{-1, q} (\Rdp)} \right\}, \\
\endaligned
\end{equation}
 and  $p \in L^q(\Rdp; \C ) + \hW^{1, q}(\Rdp; \C )$, 
where $C$ depends on $d$, $q$ and $\theta$.
\end{thm}

Our proof of Theorem \ref{H-thm-1} follows closely a line of argument in \cite{Sohr-1994}.

For a function $h$ in $\R^{d-1}$, we use  $\widehat{h}$ to denote the Fourier transform of $h$, 
\begin{equation}\label{FT}
\widehat{h} (\xi^\prime )
=\int_{\R^{d-1}} e^{-i \xi^\prime \cdot x^\prime} h(x^\prime)\, dx^\prime,
\end{equation}
for $\xi^\prime\in \R^{d-1}$.

\begin{lemma}\label{H-Lemma}
Let $T$ be a bounded linear operator on $L^2(\mathbb{R}^{d-1}; \C^m)$.
Suppose that 
$
\widehat{Tf}(\xi^\prime)=m(\xi^\prime) \widehat{f} (\xi^\prime)
$
and  that the multiplier $m(\xi^\prime)$ satisfies  the estimate,
\begin{equation}\label{HM}
|\xi^\prime|^{|\alpha|} |D^\alpha  m(\xi^\prime)|\le M,
\end{equation}
for $|\alpha|\le \left[ \frac{d-1}{2} \right] +1$, where $\alpha=(\alpha_1, \dots, \alpha_{d-1} )$ and 
 $D^\alpha =\partial_1^{\alpha_1} \cdots \partial_{d-1}^{\alpha_{d-1}}$.
Then
$$
\| Tf \|_{L^q(\mathbb{R}^{d-1})} \le C M \| f \|_{L^q(\mathbb{R}^{d-1})}
$$
for $1< q< \infty$, where $C$ depends on $d$ and  $q$.
\end{lemma}

\begin{proof}
This is the well known Mikhlin multiplier theorem in $\R^{d-1}$.
\end{proof}

We use $W^{1-\frac{1}{q}, q} (\R^{d-1}; \C^m)$ to denote  the trace space of $W^{1, q}(\Rdp; \C^m)$ on $\R^{d-1}$.

\begin{lemma}\label{Lemma-HP}
Let  $T $ be a bounded linear operator from  $L^2(\mathbb{R}^{d-1})$ to $L^2(\Rdp)$.
Suppose that 
$$
\widehat{T f}(\xi^\prime, x_d)=m (\xi^\prime, x_d)  \widehat{f}  (\xi^\prime)
$$
 and  that  $m (\xi^\prime, x_d)$ satisfies the condition
\begin{equation}\label{HM-1}
|\xi^\prime|^{|\alpha|} | D^\alpha  m(\xi^\prime, x_d)|
+ |\xi^\prime |^{|\alpha|-1} |D^\alpha \partial_d m(\xi^\prime, x_d)|
\le \frac{M_0 e^{-\delta |\xi^\prime| x_d}}{1+x_d}
\end{equation}
for $x_d>0$, $\xi^\prime\in \mathbb{R}^{d-1}$ and $|\alpha|\le \left[ \frac{d-1}{2} \right] +1$, where $\delta >0$.
Then
\begin{equation}\label{HP-2}
\left\{
\aligned
\| T f \|_{L^{ q} (\Rdp)}
 & \le C  \| f \|_{L^ q(\mathbb{R}^{d-1})}, \\
 \| \nabla T(f) \|_{L^q (\Rdp)}
  & \le C \| f \|_{W^{1-\frac{1}{q}, q} (\R^{d-1})},
 \endaligned
 \right.
\end{equation}
for $1<q<\infty$, where $C$ depends on $d$, $q$, $\delta$ and $M_0$.
\end{lemma}

\begin{proof}
Note that for each $x_d>0$, $m(\xi^\prime, x_d)$ satisfies \eqref{HM} with $M=M_0(1+x_d)^{-1}$.
It follows from Lemma \ref{H-Lemma} that
$$
\aligned
\int_{\Rdp} |Tf (x^\prime, x_d)|^q\, dx^\prime dx_d
&\le C M_0^q \int_0^\infty \int_{\R^{d-1}} \frac{|f(x^\prime)|^q }{(  1+x_d )^q} \, dx^\prime dx_d\\
& \le CM_0^q \int_{\R^{d-1}} |f|^q \, dx^\prime.
\endaligned
$$
To prove the second inequality in \eqref{HP-2}, we write 
$$
\widehat{\partial_j Tf } (\xi^\prime, x_d)
=e^{\delta_0 x_d |\xi^\prime |} m (\xi^\prime, x_d) \cdot i \xi_j e^{-\delta_0 x_d |\xi^\prime|} \widehat{f} (\xi^\prime)
$$
for $1\le j \le d-1$, and
$$
\widehat{\partial_d Tf} (\xi^\prime, x_d)
= |\xi^\prime|^{-1} e^{\delta_0 x_d |\xi^\prime |} \partial_d m(\xi^\prime, x_d) 
\cdot |\xi^\prime| e^{-\delta_0 x_d |\xi^\prime |} \widehat{f} (\xi^\prime),
$$
where $\delta_0=\delta/2$.
Using \eqref{HM-1}, it is not hard to show that  for each $x_d>0$,  both 
$$
e^{\delta_0 x_d |\xi^\prime |} m(\xi^\prime, x_d) \quad \text{ and } \quad
|\xi^\prime|^{-1} e^{\delta_0 x_d |\xi^\prime|}\partial_d m(\xi^\prime, x_d)
$$
satisfy the condition \eqref{HM} with $M$ independent of $x_d$.
This implies that
$$
\int_{\Rdp} | \nabla T f(x^\prime, x_d) |^q\, dx
\le C \int_{\Rdp} | \nabla v(x^\prime, x_d)|^q\, dx,
$$
where $v$ is defined by 
$$
\widehat{v}(\xi^\prime, x_d) = e^{-\delta_0 x_d |\xi^\prime|} \widehat{f}(\xi^\prime).
$$

Finally, we note that if $\delta_0=1$,  $v$ is a solution of the Dirichlet problem,
\begin{equation}\label{PP}
\left\{
\aligned
\left( \partial_1^2 +\cdots +\partial_{d-1}^2 +  \partial_d^2 \right) v  & =0 &\quad & \text{ in } \Rdp,\\
v & =f & \quad & \text{ on } \mathbb{R}^{d-1} \times \{ 0\},
\endaligned
\right.
\end{equation}
given by the Poisson integral of $f$.
It is well known that $v$  satisfies  the estimate,
$$
\| \nabla  v \|_{L^q (\Rdp)} \le C \| f \|_{W^{1-\frac{1}{q}, q} (\R^{d-1})},
$$
where $C$ depends on $d$ and $q$ \cite[Chapter V]{Stein}.
The general case follows from the case $\delta_0=1$ by a rescaling in $x_d$.
As a result, we obtain the second inequality in \eqref{HP-2}.
\end{proof}

\begin{proof}[\bf Proof of Theorem \ref{H-thm-1}]

By rescaling we may assume $|\lambda|=1$.

Step 1. We establish the existence and the estimate \eqref{H-est-1}.

Let $F\in L^q(\Rdp; \Cd)$, $f\in L^q(\Rdp;  \Cdd)$ and $g\in L^q(\Rdp; \C)\cap \hW^{-1, q}( \Rdp; \C)$.
We extend $F, f, g$ to $\Rd$ by either  the even or odd reflection  in such a way that 
the extensions  satisfy the condition \eqref{ref}. 
Let $\widetilde{F}$, $\widetilde{f}$, $\widetilde{g}$ denote the extensions of $F$, $f$, $g$, respectively.
Note that $\widetilde{g}\in \hW^{-1, q}(\Rd)$ and
$$
\| \widetilde{g} \|_{\hW^{-1, q}(\R^d)}
\le 2 \| g \|_{\hW^{-1, q} (\Rdp)}.
$$
Let $(\widetilde{u}, \widetilde{p})$ denote the solution of \eqref{R-eq} in $\Rd$, given by Theorem \ref{R-thm-1}, with
data $\widetilde{F}, \widetilde{f}, \widetilde{g}$.
By Remark \ref{remark-R}, we have $\widetilde{u}_d (x^\prime, 0)=0$ for any $x^\prime \in \mathbb{R}^{d-1}$.
By subtracting $(\widetilde{u}, \widetilde{p}) $ from solutions of \eqref{H-eq}, we reduce the problem to  the 
Dirichlet problem,
\begin{equation}\label{H-1}
\left\{
\aligned
-\Delta u+\nabla p +\lambda u & = 0 & \quad & \text{ in } \Rdp,\\
\text{\rm div} (u) & =0 & \quad & \text{ in } \Rdp,\\
u_j & =h_j  & \quad & \text{ on } \mathbb{R}^{d-1} \times \{ 0 \} \text{ for } 1\le j \le d-1,\\
u_d & =0 & \quad & \text{ on } \mathbb{R}^{d-1} \times \{ 0 \},
\endaligned
\right.
\end{equation}
where $h_j= - \widetilde{u}_j$ for $1\le j \le d-1$.
We will show that there exist $u\in W^{1, q}(\Rdp; \Cd)$ and $p\in L^q(\Rdp; \C)$ such that $(u, p)$ satisfies  \eqref{H-1}  and the estimate,
\begin{equation}\label{H-2}
\| \nabla u \|_{L^q(\Rdp)} + \| u \|_{L^q(\Rdp)}
+ \| p \|_{L^q(\Rdp)} 
\le C \| h\|_{W^{1-\frac{1}{q}, q } (\mathbb{R}^{d-1})}.
\end{equation}
Since $h_j=-\widetilde{u}_j$ on $\R^{d-1}\times \{ 0\}$ and 
$$
\aligned
\| \widetilde{u}\|_{W^{1-\frac{1}{q}, q} (\mathbb{R}^{d-1})}
 & \le C \left\{ \| \nabla \widetilde{u} \|_{L^q(\Rd)} + \| \widetilde{u }\|_{L^q(\Rd)} \right\}\\
 & \le C \left\{ \| \widetilde{F}\|_{L^q(\Rd)}
 + \|\widetilde{f} \|_{L^q(\Rd)}
 + \| \widetilde{g} \|_{L^q(\Rd)}
 + \| \widetilde{g}  \|_{\hW^{-1, q}(\Rd)} \right\} \\
 &  \le C \left\{ \| {F}\|_{L^q(\Rdp)}
 + \| f\|_{L^q(\Rdp)}
 + \| g \|_{L^q(\Rdp)}
 + \| g \|_{\hW^{-1, q}(\Rdp)} \right\}, \\
\endaligned
$$
the desired estimate \eqref{H-est-1}  follows from \eqref{H-2}.

To solve \eqref{H-1}, we use the partial Fourier transform in $x^\prime=(x_1, \dots, x_{d-1})$, defined by \eqref{FT}.
Let 
$$
m_0 (s, x_d) =\frac{e^{-\sqrt{\lambda + s^2} x_d} -e^{-s x_d}}{\sqrt{\lambda + s^2} -s },
$$
where $s=|\xi^\prime|$.
It follows from \cite{Sohr-1994} that 
a solution of  \eqref{H-1} in the partial Fourier transform  is given by
\begin{equation}\label{H-1-1}
\left\{
\aligned
\widehat{u}_j (\xi^\prime, x_d) & =
-\partial_d m_0(s, x_d) \frac{\xi_j \xi_k}{s^2} \widehat{h}_k (\xi^\prime)
+ \left( \delta_{jk} -\frac{\xi_j \xi_k}{s^2} \right)
e^{-\sqrt{\lambda +s^2} x_d} \widehat{h}_k (\xi^\prime)\\
\widehat{u}_d (\xi^\prime, x_d) & =i m_0 (s, x_d) \xi_k \widehat{h}_k (\xi^\prime)
\endaligned
\right.
\end{equation}
for $1\le j \le d-1$, and
\begin{equation}\label{H-1-1p}
\widehat{p} (\xi^\prime, x_d)
=-s^{-2} (\lambda +s^2 -\partial_d^2) \partial_d \widehat{u}_d,
\end{equation}
where the repeated index $k$ is summed from $1$ to $d-1$.
Write 
$$
\widehat{u}_j (\xi^\prime, x_d) =m_{jk} (\xi^\prime, x_d) \widehat{h}_k(\xi)
$$
 for $1\le j \le d$.
 Note that  $m_{jk}$ satisfies the condition \eqref{HM-1} (see \cite[Lemma 2.5]{Sohr-1994}).
 By Lemma \ref{Lemma-HP}, we obtain
 $$
\| u \|_{L^q(\Rdp) } + \|\nabla u \|_{L^q(\Rdp)} \le C \| h \|_{W^{1-\frac{1}{p}, p} (\mathbb{R}^{d-1})}
 $$
 for $1<q<\infty$. 
Using the fact that
$$
s^{-2} (\lambda +s^2-\partial_d^2) \partial_d m_0 (s, x_d)=
s^{-1} (\sqrt{\lambda+ s^2} + s) e^{-s x_d }, 
$$
and that $s^{-1} (\sqrt{\lambda +s^2} +s ) e^{- s x_d}$ satisfies the condition \eqref{HM-1}, it follows again by Lemma \ref{Lemma-HP} that 
$$
\int_{\Rdp} |p|^q\, dx
\le C   \| h \|^q_{W^{1-\frac{1}{q}, q} (\mathbb{R}^{d-1}) }.
$$
As a result, we have proved \eqref{H-2}.

\medskip

Step 2. With the existence established in Step 1 at our disposal,
 the uniqueness may be proved by using the same argument as in the proof of Theorem \ref{R-thm-1}.
 We omit the details.
%Let  $u\in W_0^{1, q}(\Rdp; \Cd)$ and $p\in L^q_{loc}(\Rdp)$ be a solution of \eqref{H-eq} 
%with $F=0$, $f=0$ and $g=0$.
%It follows that if $v \in C_0^\infty(\Rdp; \Cd)$ and $ \text{\rm div}(v)=0$ in $\Rdp$, then
%\begin{equation}\label{H-9}
%\int_{\Rdp} \nabla u \cdot \nabla v
%+ \lambda \int_{\Rdp} u \cdot v  =0.
%\end{equation}
%Since $u \in W^{1, q}(\Rdp; \Cd)$, by a limiting argument, the equation above also hold for any
%$ v \in W_0^{1, q^\prime }(\Rdp; \Cd)$ with the property that  $\text{\rm div} (v )=0$ in $\Rdp$.
%We now choose $v$ to be a solution of 
%\begin{equation}\label{H-10} 
%\left\{
%\aligned
%-\Delta  v +\nabla \phi +\lambda  v  & = |u|^{q-2} \overline{u}  & \quad &  \text{ in } \Rdp,\\
%\text{\rm div} ( v ) & =0 & \quad & \text{ in } \Rdp,\\
%v  & =0 & \quad & \text{ on } \mathbb{R}^{d-1} \times \{ 0\},
%\endaligned
%\right.
%\end{equation}
%where $\overline{u}$ denotes the complex conjugate of $u$.
%Note that $|u|^{p-2} \overline{u}\in L^{q^\prime} (\Rdp; \Cd)$. 
%By Step 1, the system \eqref{H-10} has a solution $(v , \phi)\in W_0^{1, q^\prime} (\Rdp, \Cd) \times ( L^{q^\prime}(\Rdp; \C) +\hW^{1, q^\prime }(\Rdp; \C))$.
%It  follows from \eqref{H-10} by a limiting argument that
%\begin{equation}\label{H-11}
%\int_{\Rdp} \nabla  u \cdot  \nabla v
%+\lambda \int_{\Rdp} u  \cdot v =\int_{\Rdp} |u|^q, 
%\end{equation}
%which, together with \eqref{H-9}, yields $\int_{\Rdp} |u|^q=0$. As a result, $u=0$  and $p$ is constant in $\Rdp$.
\end{proof}

\begin{remark}\label{re-p-H}
Let $\lambda\in \Sigma_\theta$ and $|\lambda|=1$.
Let $(u, p)$ be the solution of \eqref{H-eq}, given by  Theorem \ref{H-thm-1}.
It follows from the proof of Theorem \ref{H-thm-1} that
$p=p_1 + p_2$, where $p_1\in L^q(\Rdp; \C)$, $p_2 \in \hW^{1, q} (\Rdp; \C)$, and
$$
\| p_1 \|_{L^q(\Rdp)}
+ \| \nabla p_2 \|_{L^q(\Rdp)}
\le C \left\{ \| F \|_{L^q(\Rdp)}
+ \| f \|_{L^q(\Rdp)}
+ \| g \|_{L^q(\Rdp)}
+ \| g \|_{\hW^{-1, q} (\Rdp)} \right\}, 
$$
where $C$ depends only on $d$, $q$ and $\theta $.
\end{remark}

\begin{remark}\label{re-H-reg}
Let $1<q_1<q_2<\infty$.
Suppose that  $F\in L^{q_j}(\Rdp; \Cd)$, $f\in L^{q_j} (\Rdp; \Cdd)$, and
$g \in L^{q_j} (\Rdp; \C)\cap \hW^{-1, q_j}(\Rdp; \C)$ for $j=1, 2$.
Let $u^j \in W^{1, q_j}_0(\Rdp; \Cd) $ be the solution of \eqref{H-eq}, given by Theorem \ref{H-thm-1},   with the same data $F, f, g$, for $j=1, 2$.
Since the solutions constructed in $W^{1, q}_0(\Rdp; \Cd)$ for the existence part of the proof  do not  depend on $q$,
it follows that    $u^1=u^2$  in $\Rdp$.
As a result, we obtain $u^1=u^2 \in W^{1, q_1}_0(\Rdp; \Cd)\cap W^{1, q_2}_0 (\Rdp; \Cd)$.
\end{remark}

%%%%%%%%%%%%%%%%%%%%%%%%

\section{The region above a Lipschitz graph}\label{section-G}

Let 
$$
\H_\psi= \left\{ (x^\prime, x_d) \in \Rd: \  x^\prime\in \mathbb{R}^{d-1} \text{ and } x_d > \psi (x^\prime) \right\}, 
$$
where $\psi: \mathbb{R}^{d-1} \to \mathbb{R}$ is a Lipschitz function.
Note that if $\psi=0$, we have $\H_0 =\Rdp$.
In this section we study  the resolvent problem for the Stokes equations,
\begin{equation}\label{G-eq}
\left\{
\aligned
-\Delta u +\nabla p+\lambda u  & =F +\text{\rm div} (f) & \quad & \text{ in } \H_\psi,\\
\text{\rm div}(u) & = g & \quad & \text{ in } \H_\psi,\\
u & =0 & \quad  & \text{ on } \partial \H_\psi,
\endaligned
\right.
\end{equation}
 where $\lambda\in \Sigma_\theta$.
For $1< q< \infty$, define
\begin{equation}\label{AB}
\aligned
A_\psi^q =L^q(\H_\psi; \C)  + \hW^{1, q}(\H_\psi; \C) \quad
\text{ and } \quad
B_\psi^q = L^q(\H_\psi; \C) \cap \hW^{-1, q}(\H_\psi; \C),
\endaligned
\end{equation}
where, as in the case $\Rd$ and $\Rdp$, 
$$
\hW^{1, q}(\H_\psi; \C)=\left\{  u\in L^q_{\loc}({\H_\psi}; \C): \ \nabla u \in L^q(\H_\psi; \Cd ) \right\},
$$
with the norm $\|\nabla u \|_{L^q(\H_\psi)}$, 
and $\hW^{-1, q}(\H_\psi; \C)$ denotes the dual of $\hW^{1, q^\prime} (\H_\psi;  \C)$.
Note that $A_\psi^q$ and $B_\psi^q$ are Banach spaces with the usual norms,
$$
\| p \|_{A_\psi^q}
=\inf \left\{ \| p_1 \|_{L^q(\H_\psi)} + \| \nabla p_2 \|_{L^q(\H_\psi)}: \ \  p=p_1 + p_2  \text{ in } \H_\psi \right\}
$$
and
$$
\| g \|_{B_\psi^q}
= \| g \|_{L^q(\H_\psi)} + \| g \|_{\hW^{-1, q} (\H_\psi)}.
$$

The goal of this section is to prove the following.

\begin{thm}\label{thm-G}
Let $ \lambda\in \Sigma_\theta$ and $1< q<\infty$.
There exists $c_0\in (0, 1)$,  depending only on $d$, $q$ and $\theta$, such that 
if  $\|\nabla^\prime \psi \|_\infty\le  c_0$, then 
for any $F\in L^q(\H_\psi; \Cd)$, $f\in L^q(\H_\psi; \Cdd)$ and $g \in B_\psi^q$,
there exists  a unique  $(u, p)$ such that $u \in W_0^{1, q}(\H_\psi; \Cd)$, 
$p\in A_\psi^q $, and \eqref{G-eq} holds.
Moreover, the solution satisfies 
\begin{equation}\label{G-est}
\aligned
 & |\lambda|^{1/2} \| \nabla u \|_{L^q(\H_\psi)}
+|\lambda| \| u \|_{L^q(\H_\psi)}
\\
&\quad \le C \left\{
\| F \|_{L^q(\H_\psi)}
+ |\lambda|^{1/2}  \| f \|_{L^q(\H_\psi)}
+|\lambda|^{1/2}  \| g \|_{L^q(\H_\psi)}
+  |\lambda| \| g \|_{\hW^{-1, q}(\H_\psi)}
\right\},
\endaligned
\end{equation}
where $C$ depends only on $d$, $q$ and $\theta$.
\end{thm}

To prove Theorem \ref{thm-G}, we introduce two Banach spaces, 
\begin{equation}
X^q_\psi=W^{1, q}_0 (\H_\psi; \Cd) \times A_\psi^q
\quad \text{ and } \quad
Y_\psi^q =W^{-1, q}(\H_\psi; \Cd) \times  B_\psi^q,
\end{equation}
with the usual product norms.
For $\lambda\in \Sigma_\theta$ with $|\lambda|=1$, consider the operator 
\begin{equation}
S_\psi^\lambda (u, p)
= \big(-\Delta u  +\nabla p+\lambda u, \text{\rm div} (u) \big).
\end{equation}
It is not hard to see that $S_\psi ^\lambda$ is a bounded linear operator from $X_\psi^q$ to $Y_\psi^q$ for any $1< q< \infty$ and that 
\begin{equation}
\| S_\psi^\lambda (u, p)\|_{Y_\psi^q}
\le C \| (u, p) \|_{X_\psi^q},
\end{equation}
where $C$ depends only on $d$ and $q$.
Using Theorem \ref{H-thm-1} and a perturbation argument, 
we will show that $S_\psi^\lambda$ is invertible if $\|\nabla^\prime  \psi \|_\infty$ is sufficiently small.

\begin{lemma}\label{G-lemma}
Let $\lambda\in \Sigma_\theta$ with  $|\lambda|=1$ and $1< q< \infty$.
Assume that  $\psi=0$. Then $S_0^\lambda: X_0^q \to Y_0^q$ is a bijection and
\begin{equation}\label{G-5} 
\| (S_0^\lambda) ^{-1} \|_{Y_0^q \to X_0^q} \le C,
\end{equation}
where $C$ depends only on $d$, $q$ and $\theta$.
\end{lemma}

\begin{proof}
In the case $\psi=0$, we have $ \H_\psi =\Rdp$.
The lemma follows readily from Theorem \ref{H-thm-1} and the estimate for $p$ in Remark \ref{re-p-H}.
Indeed, note that for any $\Lambda \in W^{-1, q}(\Omega; \Cd)$, there exist 
$F\in L^q(\Omega; \Cd)$ and $ f\in L^q(\Omega; \Cdd)$ such that
$\Lambda= F +\text{\rm div}(f)$ and  $\|\Lambda\|_{W^{-1, q}(\Omega)}
\approx \| F \|_{L^q(\Omega)} + \| f \|_{L^q(\Omega)}$.
\end{proof}

\begin{thm}\label{G-thm-1}
Let $\lambda\in \Sigma_\theta$ with $|\lambda|=1$.
Let $1<q< \infty$.
There exists $c_0\in (0, 1)$, depending only on $d$, $q$ and $\theta$, such that
if $\|\nabla^\prime\psi  \|_\infty\le  c_0$,  then 
$S_\psi^\lambda: X_\psi ^q \to Y_\psi ^q$ is a bijection and
\begin{equation}\label{G-6}
\| ( S_\psi^\lambda) ^{-1} \|_{Y_\psi ^q \to X_\psi ^q} \le C,
\end{equation}
where $C$ depends only on $d$, $q$ and $\theta$.
\end{thm}

\begin{proof}

Suppose $\|\nabla^\prime \psi\|_\infty\le  1$.
Define a bi-Lipschitz map $\Psi: \H_\psi \to \Rdp$ by
$$
\Psi (x^\prime, x_d) = (x^\prime, x_d -\psi (x^\prime)).
$$
Note that $\Psi^{-1} (x^\prime, x_d) = (x^\prime, x_d +\psi (x^\prime))$.
For a function $u$  in $\H_\psi $, let $\wu =u \circ \Psi^{-1} $, defined in $\Rdp$.
Thus, $u=\wu\circ \Psi$ and 
$$
\left\{
\aligned
\partial_j u  & =\partial_j \wu \circ \Psi- \partial_d (\wu \partial_j \psi) \circ \Psi \quad \text{ for } 1\le j \le d-1,\\
\partial_d u & =\partial_d \wu \circ \Psi.
\endaligned
\right.
$$
A computation shows that 
$$
\Delta u = \Delta \wu\circ \Psi
-\partial_d ( \partial_k\wu \partial_k \psi ) \circ \Psi
-\partial_k( \partial_d \wu \partial_k \psi) \circ \Psi
+ \partial_d ( \partial_d\wu  |\nabla^\prime \psi|^2 ) \circ \Psi,
$$
where $|\nabla^\prime \psi|^2 = |\partial_1 \psi|^2 +\cdots + |\partial_{d-1} \psi|^2$ and the repeated index $k$ is summed from $1$ to $d-1$.
For $(u, p) \in X_\psi^q$, let  $\wu=u\circ \Psi^{-1} $ and $\wp=p\circ\Psi^{-1} $. Then 
\begin{equation}\label{G-10}
\aligned
-\Delta u_j +\partial_j p +\lambda u_j 
  =- & \Delta \wu_j \circ \Psi +\partial_j \wp \circ \Psi+ \lambda \wu_j\circ \Psi\\
&+\partial_d (\partial_k \wu_j \partial_k \psi) \circ \Psi
+ \partial_k (\partial_d \wu_j \partial_k \psi) \circ \Psi\\
& -\partial_d ( \partial_d \wu_j |\nabla^\prime \psi|^2 ) \circ \Psi
-\partial_d ( \wp \partial_j \psi) \circ \Psi
\endaligned
\end{equation}
for $1\le j \le d-1$, and
\begin{equation}\label{G-11}
\aligned
-\Delta u_d +\partial_dp +\lambda u_d 
  = & -\Delta \wu_d  \circ \Psi +\partial_d \wp \circ \Psi+ \lambda \wu_d \circ \Psi\\
&+\partial_d (\partial_k \wu_d \partial_k \psi) \circ \Psi
+ \partial_k (\partial_d \wu_d \partial_k \psi) \circ \Psi\\
& -\partial_d ( \partial_d \wu_d  |\nabla^\prime \psi|^2 ) \circ \Psi, 
\endaligned
\end{equation}
where the repeated index $k$ is summed from $1$ to $d-1$.
Also, note that
\begin{equation}\label{G-12}
\text{\rm div} (u) =\text{\rm div}(\wu) \circ \Psi -\partial_d ( \wu_k \partial_k \psi ) \circ \Psi.
\end{equation}
In view of \eqref{G-10}, \eqref{G-11} and \eqref{G-12}, we obtain 
\begin{equation}\label{G-13}
S^\lambda_\psi (u, p)
=S^\lambda_0 (\wu, \wp) \circ \Psi + R(\wu, \wp)\circ \Psi,
\end{equation}
where $R(\wu, \wp)= (R_1 (\wu, \wp), \dots, R_d(\wu, \wp), R_{d+1} (\wu, \wp))$ with 
\begin{equation}\label{G-14}
R_j (\wu, \wp)
=\partial_d (\partial_k \wu_j \partial_k \psi) 
+ \partial_k (\partial_d \wu_j \partial_k \psi) 
 -\partial_d ( \partial_d \wu_j |\nabla^\prime \psi|^2 ) 
-\partial_d ( \wp \partial_j \psi) \\
\end{equation}
for $1\le j \le d-1$,  and 
\begin{equation}\label{G-15}
\left\{
\aligned
R_d (\wu, \wp)
&=
\partial_d (\partial_k \wu_d \partial_k \psi) 
+ \partial_k (\partial_d \wu_d \partial_k \psi) 
-\partial_d ( \partial_d \wu_d  |\nabla^\prime \psi|^2 ), \\
R_{d+1}(\wu, \wp)
&=-\partial_d ( \wu_k \partial_k \psi ).
\endaligned
\right.
\end{equation}

We claim that  for any $(\wu, \wp)\in X_0^q$, 
\begin{equation}\label{G-16}
\| R (\wu, \wp) \|_{Y_0^q}
\le C \|\nabla^\prime \psi \|_\infty \| (\wu, \wp) \|_{X_0^q},
\end{equation}
where $C$ depends only on $d$ and $q$.
To show \eqref{G-16}, we note that 
\begin{equation}\label{G-17}
\| R_j (\wu, \wp)\|_{W^{-1, q}(\Rdp)}
\le C \| \nabla^\prime \psi \|_\infty \| \nabla \wu \|_{L^q(\Rdp)}
+ \| \partial_d (\wp \partial_j \psi) \|_{W^{-1, q} (\Rdp)}
\end{equation}
for $1\le j \le d-1$, and
\begin{equation}\label{G-18}
\| R_d (\wu, \wp)\|_{W^{-1, q}(\Rdp)}
\le C \| \nabla^\prime \psi \|_\infty \| \nabla \wu \|_{L^q(\Rdp)},
\end{equation}
where we have used the assumption $\|\nabla^\prime \psi \|_\infty \le 1$.
To bound the second term in the right-hand side of \eqref{G-17}, we let
$$
\wp=\wp_1  + \wp_2 \in  L^q(\Rdp; \C) + \hW^{1, q} (\Rdp; \C) = A_0^q.
$$
Then
$$
\aligned
 \| \partial_d (\wp \partial_j \psi) \|_{W^{-1, q} (\Rdp)}
 & \le  \| \partial_d (\wp_1  \partial_j \psi) \|_{W^{-1, q} (\Rdp)} +  \| \partial_d (\wp_2 \partial_j \psi) \|_{W^{-1, q} (\Rdp)}\\
 & \le C \| \nabla^\prime \psi\|_\infty  \| \wp_1 \|_{L^q (\Rdp)} 
 + C \|\nabla^\prime \psi \|_\infty \|\partial_d \wp_2 \|_{L^q(\Rdp)}.
 \endaligned
 $$
This shows that
$$
 \| \partial_d (\wp \partial_j \psi) \|_{W^{-1, q} (\Rdp)}
\le C \| \nabla^\prime \psi \|_\infty  \| \wp \|_{A_0^q}.
$$
As a result, we have proved that
\begin{equation}
\| R_j (\wu, \wp) \|_{W^{-1, q}(\Rdp)}
\le C \| \nabla^\prime \psi \|_\infty \| (\wu, \wp) \|_{X_0^q}
\end{equation}
for $1\le j \le d$. This, together with the estimates,
$$
\| R_{d+1} (\wu, \wp)\|_{L^q(\Rdp)}
+ \| R_{d+1} (\wu, \wp) \|_{\hW^{-1, q}(\Rdp)}
\le C \| \nabla^\prime \psi \|_\infty
\left( \|\nabla \wu \|_{L^q(\Rdp)} + \| \wu \|_{L^q(\Rdp)} \right),
$$
gives \eqref{G-16}.

By Lemma \ref{G-lemma}, $S_0^\lambda: X_0^q \to Y_0^q$ is bounded and invertible
for $1< q< \infty$.
It follows by a standard perturbation argument  that $S_0^\lambda +R: X_0^q \to Y_0^q$ is bounded and invertible if
$$
\| R (S_0^\lambda)^{-1} \|_{Y_0^q \to Y_0^q} < 1.
$$
Moreover,  we have 
$$
\| (S_0^\lambda +R)^{-1} \|_{Y_0^q \to X_0^q}
\le \frac{ \| (S_0^\lambda)^{-1} \|_{Y_0^q \to X_0^q} }{ 1- \| R (S_0^\lambda)^{-1} \|_{Y_0^q \to Y_0^q}}.
$$
By \eqref{G-5} and  \eqref{G-16}, 
$$
\aligned
\| R (S_0^\lambda)^{-1} \|_{Y_0^q \to Y_0^q}
 & \le \| R \|_{X_0^q \to Y_0^q}   \| (S_0^\lambda)^{-1} \|_{Y_0^q \to X_0^q}\\
 &  \le C_0 \|\nabla^\prime \psi \|_\infty,
\endaligned
$$
where $C_0$ depends only on $d$, $q$ and $\theta$.
As a result, we have proved that if $\| \nabla^\prime \psi \|_\infty \le  (2C_0)^{-1}$, then 
$S_0^\lambda +R: X_0^q \to Y_0^q$ is invertible and
$$
\| (S_0^\lambda + R)^{-1} \|_{Y_0^q \to X_0^q} 
\le C
$$
for some $C$ depending on $d$, $q$ and $\theta$.
Finally, we note that 
$$
\| (u\circ \Psi^{-1} , p\circ \Psi^{-1} ) \|_{X_0^q} \approx \| (u, p) \|_{X_\psi ^q}
$$
for any $(u, p) \in X_\psi^q$, and
$$
\| (\Lambda \circ \Psi^{-1} , g \circ \Psi^{-1} )\|_{Y_0^q} \approx \| (\Lambda, g ) \|_{Y_\psi ^q}
$$
for any $(\Lambda, g ) \in Y_\psi^q$.
By \eqref{G-13}, we deduce that  if $\|\nabla^\prime\psi  \|_\infty
\le c_0(d, q, \theta)$,  then  $S_\psi^ \lambda: X_\psi^q \to Y_\psi^q$ is invertible and \eqref{G-6} holds.
This completes the proof.
\end{proof}

\begin{proof}[\bf Proof of Theorem \ref{thm-G}]
The case $|\lambda|=1$ follows readily from Theorem \ref{G-thm-1}.
The general case can be reduced to the case $|\lambda|=1$ by rescaling.
Indeed, let $(u, p)$ be a solution of \eqref{G-eq} in $\H_\psi$.
Let  $v(x)= u(|\lambda|^{-1/2} x)$ and $\phi (x) =|\lambda|^{-1/2} p ( |\lambda|^{-1/2}  x)$.
Then $(v, \phi)$ is a solution of the resolvent problem for the Stokes equations
in the graph domain  $\H_{\psi_\lambda} $ with the parameter $ \lambda |\lambda|^{-1}\in \Sigma_\theta$, where
$\psi_\lambda (x^\prime)= |\lambda|^{1/2} \psi ( |\lambda|^{-1/2}  x^\prime)$.
Moreover,  we have $\|\nabla^\prime \psi_\lambda  \|_\infty = \| \nabla^ \prime \psi \|_\infty$.
As a result, the general case follows from the case $|\lambda|=1$.
\end{proof}

\begin{remark}\label{re-G-0}
Let $1<q_1< q_2< \infty$. 
Let $\lambda\in \Sigma_\theta$ and $|\lambda|=1$.
It follows from Lemma \ref{G-lemma} and Remark \ref{re-H-reg} that 
$S_0^\lambda: X_0^{q_1} \cap X_0^{q_2}
\to Y_0^{q_1} \cap Y_0^{q_2}$ is a bijection and
$$
\| (S_0^\lambda)^{-1} \|_{Y_0^{q_1} \cap Y_0^{q_2} \to X_0^{q_1}\cap X_0^{q_2}} \le C,
$$
where $C$ depends only on $d$, $q_1$, $q_2$ and $\theta$.
By the same perturbation argument as in the proof of Theorem \ref{G-thm-1}, we deduce that 
$S_\psi ^\lambda: X_\psi ^{q_1} \cap X_\psi ^{q_2}
\to Y_\psi ^{q_1} \cap Y_\psi ^{q_2}$ is a bijection and
$$
\| (S_\psi ^\lambda)^{-1} \|_{Y_\psi ^{q_1} \cap Y_\psi ^{q_2} \to X_\psi ^{q_1}\cap X_\psi ^{q_2}} \le C,
$$
if  $\|\nabla^\prime \psi  \|_\infty\le  c_0(d,  q_1, q_2, \theta )$, where  $C$ depends only on $d$, $q_1$, $q_2$ and $\theta$.
Consequently, if $F\in L^{q_1} (\H_\psi;  \Cd)\cap L^{q_2} (\H_\psi; \Cd)$,
$f\in L^{q_1} (\H_\psi; \Cdd) \cap L^{q_2} (\H_\psi;  \Cdd)$ and $g \in B^{q_1}_\psi \cap B^{q_2}_\psi$, 
then the  solution $u$ of \eqref{G-eq},  given by Theorem \ref{thm-G}, belongs to $W^{1, q_1}_0 (\H_\psi; \Cd)\cap W^{1, q_2}_0(H_\psi; \Cd)$, 
provided that $\|\nabla^\prime \psi \|_\infty$ is sufficiently small.
\end{remark}

\begin{remark}\label{re-G}
Let $(u, p)$ be a solution of the resolvent problem for the Stokes equations in $\H_\psi$.
Let $v(x)  =O^T u (Ox)$ and $\phi (x) = p(Ox)$, where $O$ is a  $d\times d$ orthogonal matrix.
Then
$$
\left\{
\aligned
(-\Delta v +\nabla \phi +\lambda  v )(x)  & = O^T (-\Delta  u +\nabla p + \lambda  u )(Ox),\\
\text{div}(v)  (x) & =\text{\rm div} (u) (Ox).
\endaligned
\right.
$$
Consequently, Theorem \ref{thm-G} continues to hold if the domain $\H_\psi$ is replaced by 
$$
O \H_\psi=
\{ y\in \Rd: y= Ox \text{ for some } x \in \H_\psi \}
$$
for any   $d\times d$ orthogonal matrix.
\end{remark}

%%%%%%%%%%%%%%%%%%

\section{A bounded $C^1$ domain and the proof of Theorem \ref{main-1}}\label{section-B}

Throughout this section we assume that $\Omega$ is a bounded $C^1$ domain in $\Rd$.
This implies  that for any $c_0>0$, there exists some $r_0>0$ such that for each $z=(z^\prime, z_d) \in \partial \Omega$, 
\begin{equation}\label{B-1}
\Omega \cap B(z, 2 r_0)= D \cap B(z, 2r_0)
\quad \text{ and } \quad
\partial \Omega \cap B(z, 2r_0) =\partial D \cap B(z, 2r_0),
\end{equation}
where  $D$ is given by
\begin{equation}\label{B-2}
D=  O \H_\psi  \quad \text{ for some orthogonal matrix $O$ and some $C^1$ function $\psi$ in $\mathbb{R}^{d-1}$} 
\end{equation}
 with $\nabla^\prime \psi(z^\prime)=0$ and $\|\nabla^\prime \psi \|_\infty\le  c_0$.
 Recall that $\nabla^\prime$ denotes the gradient with respect to $x^\prime =(x_1, \dots, x_{d-1})$.
We will use $L_0^q(\Omega; \C)$ to denote the subspace  of $L^q(\Omega; \C)$ of functions $p$ with $\int_\Omega p=0$.

The goal of this section is to prove the following theorem, which contains Theorem \ref{main-1} as 
a special case with  $f=0$ and $g=0$.

\begin{thm}\label{thm-B-1}
Let $\Omega$ be a bounded $C^1$ domain in $\Rd$, $d\ge 2$.
Let $1< q< \infty$ and $\lambda\in \Sigma_\theta$.
For any $F\in L^q(\Omega; \Cd)$,  $f\in L^q(\Omega; \Cdd)$ and $g \in L_0^q(\Omega; \C)$, 
there exists a unique $u \in W^{1, q}_0(\Omega; \Cd)$ such that
\begin{equation}\label{B-eq-0}
\left\{
\aligned
-\Delta u +\nabla p +\lambda u & = F +\text{\rm div}(f),\\
\text{\rm div}(u) & =g
\endaligned
\right.
\end{equation}
hold in $\Omega$ for some $p\in L^1_{\loc}(\Omega; \C)$ in the sense of distributions.
Moreover, the solution $u$ satisfies the estimate, 
\begin{equation}\label{B-est-0}
\aligned
 & (|\lambda| +1)^{1/2} \| \nabla u \|_{L^q(\Omega)}
+ (|\lambda| +1) \| u \|_{L^q(\Omega)}\\
&\qquad\qquad\qquad
\le C \left\{ \| F \|_{L^q(\Omega)}
+ (|\lambda| +1)^{1/2} \| f \|_{L^q(\Omega)} 
+ (|\lambda| +1) \| g \|_{L^q(\Omega)} \right\},
\endaligned
\end{equation}
and $p\in L^q(\Omega; \C)$,
where $C$ depends only on $d$, $q$, $\theta$ and $\Omega$.
\end{thm}

Theorem \ref{thm-B-1} follows from  Theorems \ref{R-thm-1}  and \ref{thm-G} by  a localization argument.

\begin{lemma}\label {Lemma-B-1}
Let $u\in W^{1, q}_0(\Omega; \Cd)$ for some $1< q< \infty$. Suppose  $\text{\rm div} (u)=0$ in $\Omega$. Then
\begin{equation}\label{B-1-0}
\| \text{\rm div}(u \varphi)\|_{\hW^{-1, q}(\Rd)}
\le C  (\|\nabla \varphi\|_\infty + \| \nabla^2 \varphi \|_\infty) \| u \|_{W^{-1, q}(\Omega)},
\end{equation}
where  $\varphi \in C_0^\infty (\Omega; \R)$
and $C$ depends on $d$, $q$,  diam$(\Omega)$ and the Lipschitz character of $\Omega$.
\end{lemma}

\begin{proof}
Let $h\in \hW^{1, q^\prime }(\Rd; \C)$.
Note that
$$
\int_{\Rd} \text{\rm div}(u\varphi) \cdot h
  =\int_\Omega (u \cdot \nabla \varphi ) \left( h -\fint_\Omega h\right), 
$$
where we have used the assumption $\text{\rm div}(u)=0$ in $\Omega$.
It follows that 
$$
\aligned
\left| \int_{\Rd} \text{\rm div}(u\varphi) \cdot h \right|
 & \le \| u \|_{W^{-1, q}(\Omega)}
\| \nabla \varphi (h-\fint_\Omega h) \|_{W_0^{1, q^\prime}(\Omega)}\\
&\le C ( \|\nabla \varphi \|_\infty
+ \| \nabla^2 \varphi \|_\infty) \| u \|_{W^{-1, q}(\Omega)} \| \nabla h \|_{L^{q^\prime} (\Rd)},
\endaligned
$$
where we have used a Poincar\'e inequality in $\Omega$. This gives \eqref{B-1-0}.
\end{proof}

\begin{remark}\label{re-B-1}
Let $u$ be the same as in Lemma  \ref{Lemma-B-1}.
Suppose $\varphi\in C_0^\infty (B(z, 2r_0); \R)$, where $z\in \partial \Omega$ and $\Omega \cap B(z, 2r_0)$ satisfies 
\eqref{B-1}-\eqref{B-2}.
Let $W^{-1, q}_0(\Omega; \C^d)$ denote the dual of $W^{1, q^\prime}(\Omega; \C^d)$.
Then
\begin{equation}\label{B-1-00}
\| \text{\rm div} (u \varphi)\|_{\hW^{-1, q}(D)}
\le C  (\|\nabla \varphi\|_\infty + \| \nabla^2 \varphi \|_\infty) \| u \|_{W_0^{-1, q}(\Omega)},
\end{equation}
where $D$ is given by \eqref{B-1}-\eqref{B-2}.
To see this, we note that for any $h\in \hW^{1, q^\prime}(D; \C)$, 
$$
\int_{D} \text{\rm div}(u\varphi) \cdot h
  =\int_\Omega (u \cdot \nabla \varphi ) \left( h -\fint_\Omega h\right), 
$$
where we have used the assumptions that $\text{\rm div}(u)=0$ in $\Omega$ and $u=0$ on $\partial \Omega$.
\end{remark}

\begin{lemma}\label{Lemma-B-2}
Let $1<q<\infty$. Then for any $p\in L_0^q(\Omega; \C)$, 
\begin{equation}\label{B-2-0}
\| p   \|_{L^q(\Omega)}
\le C \|  \nabla  p \|_{W^{-1, q}(\Omega)}, 
\end{equation}
where $C$ depends on $d$, $q$, diam$(\Omega)$ and the Lipschitz character of $\Omega$.
\end{lemma}

\begin{proof}

Since $\Omega$ is a bounded Lipschitz domain and $\overline{p} |p|^{q-2} \in L^{q^\prime}(\Omega; \C)$,
 there exists $v\in W^{1, q^\prime}_0 (\Omega; \Cd)$  such that
$$
\text{\rm div} (v)=
\overline{p}  | p|^{q-2} -\fint_\Omega \overline{p} |p|^{q-2} \quad \text{ in } \Omega
$$
(see \cite[Theorem III.3.1]{Galdi}).
Moreover, the function $v$ satisfies 
\begin{equation}\label{B-2-1}
\| v \|_{W^{1, q^\prime}(\Omega)}
\le C \| \overline{p} |p|^{q-2}\|_{L^{q^\prime}(\Omega)}
= C \| p \|_{L^{q}(\Omega)}^{q-1}.
\end{equation}
Using 
$$
\int_\Omega |p|^q
=\int_\Omega p\cdot  \text{\rm div} (v), 
$$
we obtain 
$$
\aligned
\|p\|_{L^q(\Omega)}^q
&\le \| \nabla p \|_{W^{-1, q}(\Omega)} \| v \|_{W_0^{1, q^\prime}(\Omega)}\\
& \le C \| \nabla p \|_{W^{-1, q}(\Omega)}
\| p \|_{L^q(\Omega)}^{q-1},
\endaligned
$$
where we have used \eqref{B-2-1} for the last inequality.
This yields \eqref{B-2-0}.
\end{proof}

The following lemma contains a key a priori estimate.
Recall that $W_0^{-1, q}(\Omega; \C^d)$ denotes the dual of $W^{1, q^\prime} (\Omega; \C^d)$.

\begin{lemma}\label{Lemma-B-3}
Let $1< q< \infty$ and $\lambda\in \Sigma_\theta $.
Let $(u, p)\in W^{1, q}_0 (\Omega; \Cd) \times L_0^q(\Omega; \C)$ be a solution of \eqref{B-eq-0}
with $F\in  L^q(\Omega; \Cd)$,  $f\in L^q(\Omega; \Cdd)$ and $g=0$.
There exist $\lambda_0>1$ and $C>0$, depending only on $d$, $q$, $\theta$, diam$(\Omega)$ and the $C^1$ character of  $\Omega$,
such that   if $|\lambda|\ge \lambda_0$, then 
\begin{equation}\label{B-3-00}
|\lambda|^{1/2} \| \nabla u \|_{L^q(\Omega)}
+ |\lambda| \| u \|_{L^q(\Omega)}
\le C \left\{
\| F \|_{L^q(\Omega)} + |\lambda|^{1/2} \| f \|_{L^q(\Omega)} 
+ |\lambda | \| u \|_{W_0^{-1, q}(\Omega)} \right\}.
\end{equation}
\end{lemma}

\begin{proof}

Let  $z\in \overline{\Omega}$ and $r_0>0$ be small.
Let $\varphi \in C_0^\infty( B(z, 2r_0); \R)$ such that
$\varphi =1$ in $B(z, r_0)$ and
$|\nabla \varphi |\le C r_0^{-1}$, $|\nabla^2 \varphi|\le C r_0^{-2}$.
A computation shows that
\begin{equation}\label{L-eq}
\left\{
\aligned
-\Delta (u \varphi) +\nabla (p \varphi) + \lambda u \varphi
 & =F \varphi + \text{\rm div}(f \varphi) - f (\nabla \varphi) + p \nabla \varphi
-2\,  \text{\rm div} (u \otimes \nabla\varphi) + u \Delta \varphi,\\
\text{\rm div} (u \varphi)
 & =u \cdot \nabla \varphi.
\endaligned
\right.
\end{equation}
We consider two cases: (1) $B(z, 2r_0)\subset \Omega$ and (2) $z\in \partial\Omega$.

Case (1). Suppose $B(z, 2r_0)\subset \Omega$.
Then  the Stokes equations in  \eqref{L-eq} hold in $\Rd$.
Since $u\varphi \in W^{1, q}(\Rd; \Cd)$ and $p\varphi \in L^q(\Rd; \C)$, 
 it  follows by  Theorem \ref{R-thm-1} that 
$$
\aligned
 & |\lambda|^{1/2}  \| \nabla (u \varphi) \|_{L^q(\Rd)}
+ |\lambda| \| u \varphi \|_{L^q(\Rd)}\\
& \le C \Big\{ 
\| F \varphi \|_{L^q(\Rd)} + |\lambda|^{1/2} \| f \varphi\|_{L^q(\Rd)} + \| f \nabla \varphi \|_{L^q(\Rd)} 
+ \| p \nabla \varphi \|_{L^q(\Rd)}\\
& \qquad \qquad\qquad
+ \| u \Delta \varphi \|_{L^q(\Rd)}
+ |\lambda|^{1/2} \| u \nabla \varphi \|_{L^q(\Rd)}
+  |\lambda| \| \text{\rm div}(u \varphi)  \|_{\hW^{-1, q} (\Rd)} \Big\}.
\endaligned
$$
This leads to 
\begin{equation}\label{B-3-0}
\aligned
& |\lambda|^{1/2} \| \nabla u \|_{L^q(B(z, r_0))} +|\lambda| \| u \|_{L^q(B(z, r_0))}\\
&\le C r_0^{-2}
\Big\{ \| F \|_{L^q(\Omega)}
+(1+|\lambda|^{1/2}) \| f \|_{L^q(\Omega)} \\
& \qquad\qquad\qquad
+ \| p \|_{L^q(\Omega)} 
+(1+ |\lambda|^{1/2} ) \| u \|_{L^q(\Omega)}
+ |\lambda| \| u \|_{W^{-1, q}(\Omega)} \Big\},
\endaligned
\end{equation}
where we have used Lemma \ref{Lemma-B-1} and  the fact $\varphi=1$ in $B(z, r_0)$.

Case (2).  Suppose $z\in \partial \Omega$.
Let $D$ be given by \eqref{B-1}-\eqref{B-2}.
We assume $r_0$ is sufficiently small so that $\|\nabla^\prime \psi \|_\infty< c_0$, where $c_0=c_0(d, q, \theta)>0$
is given by Theorem \ref{thm-G}.
Note that  $u\varphi\in W^{1, q}_0(D; \Cd)$, $p\varphi \in L^q(D; \C)$, and 
\eqref{L-eq} holds in $D$.
It follows by Theorem \ref{thm-G} and Remark \ref{re-G} that
$$
\aligned
 & |\lambda|^{1/2} \| \nabla (u \varphi ) \|_{L^q(D)}
+ |\lambda| \| u \varphi \|_{L^q(D)}\\
 & \le C \Big\{
\| F \varphi \|_{L^q(D)} +  |\lambda|^{1/2} \| f\varphi \|_{L^q(D)} + \| f \nabla \varphi \|_{L^q(D)}  
+ \| p \nabla \varphi\|_{L^q(D)} + \| u \Delta \varphi\|_{L^q(D)}\\
& \qquad \qquad + |\lambda|^{1/2} \| u \nabla \varphi \|_{L^q(D)}
+ |\lambda| \| \text{div}(u \varphi)\|_{\hW^{-1, q}(D)} \Big\},
\endaligned
$$
which yields
\begin{equation}\label{B-3-1}
\aligned
& |\lambda|^{1/2} \| \nabla u \|_{L^q(\Omega \cap B(z, r_0))} +|\lambda| \| u \|_{L^q(\Omega \cap B(z, r_0))}\\
&\quad \le C r_0^{-2}
\Big\{ \| F \|_{L^q(\Omega)} + (1+ |\lambda|^{1/2} ) \| f\|_{L^q(\Omega)} 
+ \| p \|_{L^q(\Omega)} \\
&\qquad\qquad\qquad
+(1+ |\lambda|^{1/2} ) \| u \|_{L^q(\Omega)}
+ |\lambda| \| u \|_{W_0^{-1, q}(\Omega)} \Big\},
\endaligned
\end{equation}
where we have used the estimate in Remark \ref{re-B-1} and  the fact $\varphi=1$ in $B(z, r_0)$.

We now cover $\Omega$ by a finite number of balls $\{ B(z_k, r_0)\}$ with the properties that either $B(z_k, 2r_0)\subset \Omega$ or
$z_k \in \partial \Omega$.
In view of \eqref{B-3-0} and \eqref{B-3-1}, by summation, we deduce that 
\begin{equation}\label{B-3-2}
\aligned
& |\lambda|^{1/2} \| \nabla u \|_{L^q(\Omega)} +|\lambda| \| u \|_{L^q(\Omega)}\\
&\le C 
\left\{ \| F \|_{L^q(\Omega)} + (1+|\lambda| ^{1/2} ) \| f \|_{L^q(\Omega)} 
+ \| p \|_{L^q(\Omega)} 
+ (1+|\lambda| ^{1/2 }) \| u \|_{L^q(\Omega)}
+ |\lambda| \| u \|_{W_0^{-1, q}(\Omega)} \right\}\\
&\le C 
\left\{ \| F \|_{L^q(\Omega)}+ (1+|\lambda|^{1/2} ) \| f \|_{L^q(\Omega)}  +\| \nabla u \|_{L^q(\Omega)}  
+(1+ |\lambda|^{1/2})  \| u \|_{L^q(\Omega)}
+  |\lambda|  \| u \|_{W_0^{-1, q}(\Omega)} \right\},
\endaligned
\end{equation}
where we have used Lemma \ref{Lemma-B-2} 
and  the equation $\nabla p=\Delta u -\lambda u +F+\text{\rm div}(f) $
for the last  inequality.
 The constant $C$ in \eqref{B-3-2}  depends only on $d$, $q$, $\theta$ and $\Omega$.
We obtain \eqref{B-3-00} by choosing $\lambda_0>1$ so large that $|\lambda| \ge 4C |\lambda|^{1/2}$ for $|\lambda|\ge \lambda_0$.
\end{proof}

\begin{lemma}\label{Lemma-B-4}
Let $2\le  q< \infty$ and   $\lambda\in \Sigma_\theta $.
Let $(u, p)\in W^{1, q}_0 (\Omega; \Cd) \times L_0^q(\Omega; \C)$ be a solution of \eqref{B-eq-0}
with $F\in L^q(\Omega; \Cd)$,  $f\in L^q(\Omega; \Cdd)$ and $g=0$.
Then,
\begin{equation}\label{B-4-0}
(|\lambda|+1) ^{1/2} \| \nabla u \|_{L^q(\Omega)}
+ (|\lambda|+1)  \| u \|_{L^q(\Omega)}
\le C \left\{ 
\| F \|_{L^q(\Omega)} + (|\lambda|+1)^{1/2} \| f \|_{L^q(\Omega)} \right\}, 
\end{equation}
where $C>0$ depends on $d$, $q$, $\theta $,  diam$(\Omega)$ and the $C^1$ character of $\Omega$.
\end{lemma}

\begin{proof}
The case $q=2$ is well known and follows from the energy estimates.
For $q>2$, we  first consider the case $|\lambda|\ge \lambda_0$,
where $\lambda_0>1$ is given by Lemma \ref{Lemma-B-3}.
Since $\Omega$ is bounded,  by Lemma \ref{Lemma-B-3}, the estimate
\begin{equation}\label{B-4-1}
|\lambda|^{1/2} \|\nabla u \|_{L^s (\Omega)}
+ |\lambda| \| u \|_{L^s(\Omega)}
\le C \left\{ \| F \|_{L^s(\Omega)} + |\lambda|^{1/2} \| f \|_{L^s(\Omega)} 
+  |\lambda| \| u\|_{W_0^{-1, s} (\Omega)} \right\}
\end{equation}
holds for any $s\in [2, q]$.
By Sobolev imbedding, $L^t (\Omega; \C^d) \subset W_0^{-1, s}(\Omega; \Cd)$, where
$1<t<d$ and $\frac{1}{t }=\frac{1}{ s} +\frac{1}{d}$.
In particular,  if $2< s\le  \frac{2d}{d-2}$, then $L^2(\Omega;  \Cd) \subset W_0^{-1, s}(\Omega; \Cd)$ and 
$$
\aligned
|\lambda| \| u \|_{W_0^{-1, s}(\Omega)}
 & \le C |\lambda| \| u \|_{L^2(\Omega)}\\
&\le C\left\{  \| F \|_{L^2(\Omega)} + |\lambda|^{1/2} \| f\|_{L^2(\Omega)} \right\} \\
&\le C\left\{  \| F \|_{L^s(\Omega)} + |\lambda|^{1/2} \| f\|_{L^s(\Omega)} \right\}.
\endaligned
$$
This, together with \eqref{B-4-1},  gives \eqref{B-4-0} for $2< q \le  \frac{2d}{d-2}$.
By a bootstrapping  argument,  one may show that  the estimate \eqref{B-4-0} holds for any $2< q< \infty$
 in a finite number of steps.

We now consider the case $|\lambda|<\lambda_0$.
We rewrite the Stokes equations as
\begin{equation}\label{B-4-2}
\left\{
\aligned
-\Delta u +\nabla p + (\lambda + 2\lambda_0) u  & = F +\text{\rm div}(f) + 2\lambda_0 u,\\
\text{\rm div} (u) & =0.
\endaligned
\right.
\end{equation}
Since $\lambda+2 \lambda_0 \in \Sigma_\theta $ and $|\lambda +2\lambda_0|> \lambda_0$,
it follows from the previous case that
\begin{equation}\label{B-4-3}
\|\nabla u \|_{L^q(\Omega)}
\le C \left\{ \| F \|_{L^q(\Omega)} +\| f\|_{L^q(\Omega)} + \| u \|_{L^q(\Omega)}\right\}.
\end{equation}
Since $W_0^{1, 2}(\Omega; \Cd)\subset L^s(\Omega; \Cd)$ for $s=\frac{2d}{d-2}$, we obtain 
$$
\aligned
\|\nabla u \|_{L^q(\Omega)}
& \le C \left\{ \| F \|_{L^q(\Omega)} + \| f\|_{L^q(\Omega)}  + \|\nabla u \|_{L^2(\Omega)} \right\}\\
& \le C \left\{ \| F \|_{L^q(\Omega)} + \| f \|_{L^q(\Omega)} \right\}
\endaligned
$$
for $2<q\le \frac{2d}{d-2}$, where we have used the estimate \eqref{B-4-0}  for  $q=2$
for the last inequality.
 As before, a bootstrapping  argument, using  \eqref{B-4-3}, gives 
$$
\aligned
\|\nabla u \|_{L^q(\Omega)} 
&\le C\left\{  \| F \|_{L^q(\Omega)} + \| f \|_{L^q(\Omega)} \right\}
\endaligned
$$
for $2<q<\infty$ in a finite number of steps.
 This, together with a Poincar\'e inequality,  yields \eqref{B-4-0} for the case $|\lambda|< \lambda_0$.
\end{proof}

We are now in a position to give the proof of Theorem \ref{thm-B-1}.

\begin{proof}[\bf Proof of Theorem \ref{thm-B-1}]

Step 1. Consider the case  $2< q<\infty$ and $g=0$.

The uniqueness follows from the case $q=2$. To show the existence and the estimate \eqref{B-est-0},
let $F\in L^q(\Omega; \Cd)$ and $f\in L^q(\Omega; \Cdd)$.
Note  that the constant $C$ in \eqref{B-4-0}
depends only on $d$, $q$, $\theta$,  the diameter of $\Omega$ as well as the $C^1$ character of $\partial \Omega$.
As a result, we may construct  a sequence of smooth domains $\{\Omega_k \}$ such that $\Omega_k \subset \Omega$ and the
estimate \eqref{B-4-0} holds in $\Omega_k$ with a constant $C$ independent of $k$.
Let  $(u^k, p^k)$ be the unique solution in $W_0^{1, 2}(\Omega_k; \Cd)
\times L^2_0(\Omega_k; \C)$ of the Stokes system  \eqref{B-eq-0} in $\Omega_k$ 
with $g=0$, $F^k$ in the place of $F$ and $f^k$ in the place of $f$, where $F^k \in C_0^\infty(\Omega_k; \Cd)$,
$f^k \in C_0^\infty(\Omega_k; \Cdd)$ and $\|F^k -F\|_{L^q(\Omega)} + \| f^k -f \|_{L^q(\Omega) }  \to 0$.
Since  $\Omega_k$ and $F^k, f^k $  are smooth, it is well known that $(u^k, p^k)\in W^{1, q}_0(\Omega_k; \C^d) \times L_0^q(\Omega_k; \C)$ \cite{Galdi}.
We extend $(u^k, p^k)$ to $\Omega$ by zero and still denote the extension by $(u^k, p^k)$.
It follows by Lemma \ref{Lemma-B-4} that 
\begin{equation}\label{B-5-1}
(|\lambda| + 1)^{1/2}
\| \nabla u^k \|_{L^q(\Omega)}
+ ( |\lambda| +1) \| u^k \|_{L^q(\Omega)}
\le C\left\{  \| F^k \|_{L^q(\Omega_k)} + (|\lambda| +1)^{1/2} \| f ^k \|_{L^q(\Omega_k)} \right\},
\end{equation}
where $C$ depends only on $d$, $q$, $\theta$ and $\Omega$.
Note that by Lemma \ref{Lemma-B-2}, $\{ p^k \}$ is bounded in $L^q(\Omega; \C)$.
By passing to a subsequence, we may assume that 
$u^k \to u$ weakly in $W^{1, q}_0(\Omega; \Cd)$ and $p^k \to p$ weakly in $L^q(\Omega; \C)$.
It is not hard to see that $(u, p)$ is a solution of \eqref{B-eq-0} in $\Omega$  with data $(F, f)$  and  $g=0$.
By letting $k \to \infty$ in  \eqref{B-5-1}, it follows that  $u$
satisfies the estimate \eqref{B-est-0}.

\medskip

Step 2. We establish the existence and estimate \eqref{B-est-0} for $1< q< 2$ and $g=0$.

For $F, G \in C_0^\infty (\Omega; \Cd)$ and $f, h \in C_0^\infty (\Omega; \Cdd)$, 
let $ (u, p), (v, \phi)  \in W^{1, 2}_0 (\Omega; \Cd)\times L^2_0(\Omega; \C)$ be  weak solutions of \eqref{B-eq-0} in $\Omega$
with data $(F, f), (G, h) $, respectively; i.e.,
$$
\left\{
\aligned
-\Delta u +\nabla p+\lambda u & =F+\text{\rm div}(f)  & \quad & \text{ in } \Omega,\\
\text{\rm div} (u) & =0 & \quad & \text{ in } \Omega,
\endaligned
\right.
$$
$$
\left\{
\aligned
-\Delta v +\nabla \phi +\lambda v & =G + \text{\rm div}(h) & \quad & \text{ in } \Omega,\\
\text{\rm div} (v) & =0 & \quad & \text{ in } \Omega.
\endaligned
\right.
$$
Note that
$$
\int_\Omega F \cdot v -\int_\Omega f \cdot \nabla v 
=\int_\Omega \nabla u \cdot \nabla v +\lambda \int_\Omega u \cdot v
=\int_\Omega G \cdot u -\int_\Omega h \cdot \nabla u.
$$
It follows that
$$
\aligned
 & \left|  \int_\Omega G \cdot u -\int_\Omega h \cdot \nabla u\right  |
  \le \| F \|_{L^q(\Omega)} \| v \|_{L^{q^\prime}(\Omega)} +  \| f \|_{L^q(\Omega)} \| \nabla v \|_{L^{q^\prime} (\Omega)} \\
& \le C (|\lambda| +1)^{-1} 
\left\{  \| F \|_{L^q(\Omega)}
+ (|\lambda| +1)^{1/2} \| f \|_{L^q(\Omega)} \right\}
\left\{ \| G \|_{L^{q^\prime} (\Omega)}  + (|\lambda| +1)^{1/2} \| h \|_{L^{q^\prime}(\Omega)} \right\},
\endaligned
$$
where we have used the estimate,  
$$
(|\lambda| +1)^{1/2}  \| \nabla  v \|_{L^{q^\prime}(\Omega)} + (|\lambda| +1 ) \| v \|_{L^{q^\prime}(\Omega)}
\le C\left\{  \| G \|_{L^{q^\prime} (\Omega)}  + ( |\lambda | +1 )^{1/2} \| h \|_{L^{q^\prime}(\Omega)} \right\},
$$
 obtained in Step 1 for $q^\prime>2$.
By duality this gives
$$
(|\lambda| +1 )^{1/2} \| \nabla u \|_{L^q(\Omega)} + (|\lambda| +1 ) \| u \|_{L^q(\Omega)}
\le C\left\{  \| F \|_{L^q(\Omega)} + (|\lambda| +1 )^{1/2} \| f \|_{L^q(\Omega)} \right\}.
$$
As a result, we have proved  the existence and the estimate \eqref{B-est-0} for $F\in C_0^\infty (\Omega; \Cd)$ and $f\in C_0^\infty (\Omega; \Cdd)$.
The general case, where $F\in L^q(\Omega; \Cd)$,  $f\in L^q(\Omega; \Cdd)$  and $g=0$,   for $1< q< 2$,   follows readily by a density argument.

\medskip

Step 3. We establish the uniqueness.

The uniqueness for $q>2$ follows from the uniqueness for $q=2$.
To handle the case $1<q<2$, let  $u\in W^{1, q}_0 (\Omega;  \Cd)$ be 
a solution of \eqref{B-eq-0} in $\Omega$ with $F=0$,  $f=0$ and $g=0$.
Since $\overline{u} |u|^{q-2}\in L^{q^\prime} (\Omega; \Cd)$, by Step 1,  there exists $(v, \phi)\in W^{1, q^\prime}_0 (\Omega; \Cd)
\times L_0^{q^\prime}(\Omega; \C)$ such that
$$
\left\{
\aligned
-\Delta v +\nabla \phi +\lambda v & =  |u|^{q-2} \overline{u}  & \quad & \text{ in } \Omega,\\
\text{\rm div}(v) & =0 & \quad &\text{ in } \Omega.
\endaligned
\right.
$$
As in the case $\Omega=\Rd$, this leads to $\int_\Omega |u|^q=0$.
Hence, $u=0$ in $\Omega$.

\medskip

Step 4. The  case $g\neq 0$.

Let $g\in L^q_0(\Omega; \C)$.  Since $\Omega$ is a bounded Lipschitz domain, 
there exists $w\in W_0^{1, q}(\Omega; \Cd)$ such that 
\begin{equation}
\text{\rm div}(w) =g \quad \text{ in } \Omega \quad  \text{ and } \quad 
\| w\|_{L^q(\Omega)} + \|\nabla w\|_{L^q(\Omega)} \le C \| g \|_{L^q(\Omega)}.
\end{equation}
By considering $\wu=u-w$, we reduce the problem to the case $g=0$.
Indeed, let $\wu$ be a solution of 
$$
\left\{ 
\aligned
-\Delta \wu + \nabla p +\lambda \wu & = F +\text{\rm div} (f+\nabla w) -\lambda w,\\
\text{\rm div}(\wu) & =0
\endaligned
\right.
$$
in $\Omega$. Then $u=\wu+ w$ is a solution of \eqref{B-eq-0}.
\end{proof}

\begin{remark}\label{re-B-2}
Let $1< q<\infty$ and $\Omega$ be a bounded $C^1$ domain in $\R^d$.
By letting $\lambda\in \R_+$ and $\lambda\to 0$ in Theorem  \ref{thm-B-1}, one may show that  for any $F\in L^q(\Omega; \Cd)$,
$f\in L^q(\Omega; \Cdd)$ and $g \in L^q_0 (\Omega; \C)$, there exists a unique $(u, p) \in W_0^{1, q}(\Omega; \Cd) \times L^q_0(\Omega; \C)$ such that 
\begin{equation}\label{eq-B-0}
\left\{
\aligned
-\Delta u +\nabla p & = F +\text{\rm div}(f),\\
\text{\rm div}(u) & =g
\endaligned
\right.
\end{equation}
in $\Omega$.
Moreover, the solution $(u, p)$ satisfies the estimate
\begin{equation}\label{W-1-q}
\| \nabla u \|_{L^q(\Omega)} + \| p\|_{L^q(\Omega)} 
\le C \left\{ \| F\|_{L^q(\Omega)} + \| f\|_{L^q(\Omega)}
+ \| g \|_{L^q(\Omega)} \right\},
\end{equation}
where $C$ depends on $d$, $q$ and $\Omega$.
The $W^{1, q}$ estimate \eqref{W-1-q}  is known for $C^1$ domains \cite {Mitrea-2004}.  
If $\Omega$ is a bounded Lipschitz domain, the estimate \eqref{W-1-q} holds  for $(3/2)-\e <q<3+\e$ if $d= 3$, and for $(4/3) -\e <q< 4+\e$ if $d=2$,
where $\e$ depends on $\Omega$ \cite{Brown-S-1995}.
If $d\ge 4$, some partial results are known \cite{Kilty-2015}.
We point out that 
the results in \cite {Brown-S-1995, Mitrea-2004, Kilty-2015} rely on the estimates for a non-homogeneous Dirichlet problem, which is solved
 by using the methods of layer potentials. 
The approach used in this paper, which  
is based on a perturbation argument, seems to be more accessible.
However, it does not work for a general Lipschitz domain.
\end{remark}

We end this section with  a localized $W^{1, q}$ estimate that will be used in the next section.

\begin{thm}\label{thm-B-2}
Let $\Omega$ be a bounded $C^1$ domain and $2<q< \infty$. 
Let $B=B(x_0, r_0)$, where $x_0\in \partial \Omega$ and $r_0>0$ is small. 
Suppose that $u  \in W^{1, 2}(2B\cap \Omega; \Cd)$, $p\in L^2(2B\cap \Omega; \C)$, and
\begin{equation}\label{reg-C}
\left\{
\aligned
-\Delta u +\nabla p  & = F +\text{\rm div} (f) & \quad & \text{ in } 2B\cap \Omega,\\
\text{\rm div} (u) & =g & \quad & \text{ in } 2B\cap \Omega,\\
u& =0& \quad & \text{  on } 2B \cap \partial \Omega,
\endaligned
\right.
\end{equation}
where  $F \in L^q(2B\cap \Omega; \Cd)$, $f\in L^q(2B\cap \Omega; \Cdd)$ 
and $g\in L^q(2B\cap \Omega; \C)$. Then $u\in W^{1, q}(B\cap \Omega; \Cd)$, $p \in L^q(B\cap \Omega; \C)$,   and
\begin{equation}\label{reg-C-1}
\aligned
 & \| \nabla u \|_{L^q(B\cap \Omega)} + \| p-\fint_{B\cap \Omega} p \|_{L^q(B\cap \Omega)}\\
&\qquad\qquad
\le C \left\{ \| F \|_{L^q(2B\cap \Omega)}
+ \| f\|_{L^q(2B\cap \Omega)} 
+\| g \|_{L^q(2B\cap \Omega)} +  \| u \|_{L^2(2B\cap \Omega)} \right\},
\endaligned
\end{equation}
where $C$ depends on $d$, $q$, $r_0$ and $\Omega$.
\end{thm}

\begin{proof}

Theorem \ref{thm-B-2} follows from the estimate \eqref{W-1-q} by a localization argument.
However, some cares are needed to handle the error term $p (\nabla \varphi)$,  introduced by the pressure $p$,
where $\varphi$ is a cut-off function.

Consider the Stokes equations \eqref{eq-B-0} with $F=0$ and $g=0$;  i.e., 
$$
-\Delta u +\nabla p =\text{\rm div}(f) \quad \text{ and } \quad \text{\rm div}(u)=0
$$
in $\Omega$. It follows from \eqref{W-1-q} that  $\| \nabla u \|_{L^q(\Omega)} \le C \| f \|_{L^q(\Omega)}$.
By Sobolev imbedding, we obtain 
$$
\| u \|_{L^s(\Omega)} \le C \| f \|_{L^q(\Omega)}, 
$$
where $\frac{1}{s} =\frac{1}{q} -\frac{1}{d}$ and $1< q<d$.
By a duality argument, as in  Step 2  in the proof of Theorem \ref{thm-B-1},   this implies that the solution of
$$
- \Delta u +\nabla p =F \quad \text{ and } \quad \text{\rm div}(u)=0
$$
in $\Omega$ satisfies the estimate, 
$$
\| \nabla u \|_{L^q(\Omega)}  + \| p \|_{L^q(\Omega)} \le C \| F \|_{L^s (\Omega)}, 
$$
where $\frac{1}{s}=\frac{1}{q}+\frac{1}{d}$ and $1< s<d$.
This observation allows us  to  improve the estimate \eqref{W-1-q} to 
\begin{equation}\label{iW}
\|\nabla u \|_{L^q(\Omega)}
+ \| p \|_{L^q(\Omega)}
\le C \left\{ \| F\|_{L^{s_*} (\Omega)} + \| f \|_{L^q(\Omega)} + \| g \|_{L^q(\Omega)} \right\},
\end{equation}
where $s_* =\max \{ 2, s\}<q$ and $\frac{1}{s}=\frac{1}{q} +\frac{1}{d}$.
Using \eqref{iW}, a standard localization procedure, together with a bootstrapping argument,  yields \eqref{reg-C-1}.
We omit the details. 
\end{proof}

%%%%%%%%%%%%%%%%%%%%%%%

\section{An exterior $C^1$ domain and the proof of Theorem \ref{main-2}}\label{section-E}

In this section we consider the case of an exterior $C^1$ domain $\Omega$; i.e., $\Omega$ is a connected open subset of $\Rd$  with compact complement 
and $C^1$ boundary.
Let $F\in L^2(\Omega; \Cd)$, $ f\in L^2(\Omega; \Cdd)$ and $\lambda\in \Sigma_\theta$.
By  the Lax-Milgram Theorem, 
there  exists a unique  $u\in W^{1, 2}_0 (\Omega; \Cd)$ such that 
\begin{equation}\label{E-eq}
\left\{
\aligned
-\Delta u +\nabla p + \lambda u & =F+\text{\rm div}(f),\\
\text{\rm div} (u) & =0
\endaligned
\right.
\end{equation}
 holds in $\Omega$ for some $p\in L^2_{\loc}(\overline{\Omega}; \C)$ in the sense of distributions. Moreover, the solution satisfies 
\begin{equation}
|\lambda|^{1/2}  \| \nabla u \|_{L^2(\Omega)}
+|\lambda| \| u \|_{L^2(\Omega)}
\le C\left\{  \| F \|_{L^2(\Omega)} +  |\lambda|^{1/2} \| f \|_{L^2(\Omega)} \right\},
\end{equation}
where $C$ depends only on $d$ and $\theta$.
We will call  $u$ the energy solution of \eqref{E-eq}.
Note that, if $F\in L^q(\Omega; \Cd)\cap L^2(\Omega; \Cd)$
and $f\in L^2(\Omega; \Cdd) \cap L^q(\Omega; \Cdd)$  for some $q>2$,
 then $(u, p) \in W^{1, q}(\Omega\cap B; \Cd)\times  L^q(\Omega\cap B; \C)$ for any ball $B$ in $\Rd$.
This follows   from the regularity theory for the Stokes equations  \eqref{eq-B-0}  in bounded $C^1$ domains.
See Theorem \ref{thm-B-2}.

Let
\begin{equation}\label{S}
\Sigma_{\theta, \delta}
=\left\{ z\in \C: |z|> \delta \text{ and } |\text{arg} (z)|< \pi -\theta \right\},
\end{equation}
where $\theta \in (0, \pi/2)$ and $\delta \in (0, 1)$.
The goal of this section is to prove the following.

\begin{thm}\label{thm-E}
Let $\Omega$ be an exterior $C^1$ domain in $\Rd$, $d\ge 2$.
Let $1<q<\infty$ and $\lambda \in \Sigma_{\theta, \delta}$.
For any $F \in L^q(\Omega; \Cd)$ and $f\in L^q(\Omega; \Cdd)$,  there exists a unique $u\in W^{1, q}_0(\Omega; \Cd)$ such that 
\eqref{E-eq} holds in $\Omega$ for some $p\in L^1_{\loc} (\Omega; \C)$. 
Moreover, the solution satisfies the estimate, 
\begin{equation}\label{E-0-0}
|\lambda|^{1/2} \| \nabla u \|_{L^q(\Omega)}
+|\lambda| \| u \|_{L^q(\Omega)}
\le C \left\{
\| F \|_{L^q(\Omega)}
+ |\lambda|^{1/2} \| f \|_{L^q(\Omega)} \right\},
\end{equation}
where $C$ depends on $d$, $q$, $\theta$, $\delta$ and $\Omega$.
\end{thm}

 Fix a large ball $B_0=B(0, 2R_0)$ such that  $\Omega\setminus B(0, R_0)
=\Rd \setminus B(0, R_0)$ and $B_0\cap \Omega$ is a bounded $C^1$ domain.

\begin{lemma}\label{Lemma-E-1}
Let $1< q< \infty$ and $\lambda \in \Sigma_{\theta} $.
Let $u\in W^{1, 2}_0(\Omega; \Cd)$ be an energy  solution of \eqref{E-eq}
with $F \in L^q(\Omega; \Cd)\cap L^2(\Omega; \Cd)$ and $f\in L^2(\Omega; \Cdd)\cap L^q(\Omega; \Cdd)$. Then $u\in W^{1, q}_0(\Omega; \Cd)$. Moreover, 
 if $|\lambda|\ge \lambda_0$,
\begin{equation}\label{E-1-0}
|\lambda|^{1/2} \| \nabla u \|_{L^q(\Omega)}
+ |\lambda| \| u \|_{L^q(\Omega)}
\le C \left\{ \| F \|_{L^q(\Omega)} + |\lambda|^{1/2}
\| f \|_{L^q(\Omega)} +  |\lambda| \| u \|_{W_0^{-1, q} (\Omega\cap 2B_0)} \right\},
\end{equation}
where $\lambda_0>1$ and  $C$ depend on $d$, $q$, $\theta $ and $\Omega$.
\end{lemma}

\begin{proof}
The proof, which uses a localization argument,  is similar to that of Lemma \ref{Lemma-B-3} for the bounded domain.
However, we need to add another case to handle the neighborhood of  $\infty$.
Choose  $\varphi \in C^\infty (\Rd; \R)$ such that $\varphi =1$ in $\Rd\setminus B(0, 2R_0)$ and
$\varphi =0$ in $B(0, R_0)$.
Then the Stokes equations in \eqref{L-eq} hold in $\Rd$. Since $p\in L^q(\Omega\cap 2B_0)$, where $B_0=B(0, 2R_0)$,
it  follows by Theorem \ref{R-thm-1}  and Remark \ref{re-R-u} that  $u\varphi \in W^{1, q}_0(\Rd; \Cd)$ and 
$$
\aligned
& | \lambda|^{1/2} \| \nabla (u \varphi) \|_{L^q(\Rd)}
+ |\lambda| \| u\varphi  \|_{L^q(\Rd)}\\
& \le C \Big\{ \| F \varphi \|_{L^q(\Rd)} + 
|\lambda|^{1/2} \| f \varphi \|_{L^q(\Rd)}
+ \| f\nabla \varphi \|_{L^q(\Rd)} 
+ \| p \nabla \varphi \|_{L^q(\Rd)}\\
&\qquad\qquad\qquad
+ \|u\Delta \varphi \|_{L^q(\Rd)}
+ |\lambda|^{1/2} \| u \nabla \varphi \|_{L^q(\Rd)}
+ |\lambda| \| \text{div} (u \varphi) \|_{\hW^{-1, q} (\Rd )} \Big\}.
\endaligned
$$
Note that the same argument as in the proof of Lemma \ref{Lemma-B-1} also yields
$$
\|\text{\rm div} (u \varphi)\|_{\hW^{-1, q}(\Rd)} \le C \| u \|_{W_0^{-1, q} (\Omega\cap B_0)}.
$$
Hence, 
\begin{equation}\label{E-1-1}
\aligned
 & |\lambda|^{1/2} \| \nabla u \|_{L^q(\Omega \setminus B_0)}
+ |\lambda| \| u \|_{L^q(\Omega \setminus B_0)}\\
& \le C \Big\{
\| F \|_{L^q(\Omega)} + |\lambda|^{1/2} \| f \|_{L^q(\Omega)} + \| f \|_{L^q(\Omega\cap B_0)} 
+ \| p \|_{L^q(\Omega\cap B_0)}\\
& \qquad \qquad
+(1+ |\lambda|^{1/2} ) \| u \|_{L^q(\Omega\cap B_0)}
+ |\lambda| \| u \|_{W_0^{-1, q}(\Omega \cap B_0)}\Big\} . \\
\endaligned
\end{equation}
Since $\Omega\cap B_0$ is a bounded $C^1$ domain, 
it follows from the proof of Lemma \ref{Lemma-B-3} that
$$
\aligned
 & |\lambda|^{1/2} \| \nabla u \|_{L^q(\Omega \cap B_0  )}
+ |\lambda| \| u \|_{L^q(\Omega \cap  B_0)}\\
& \le C \Big\{
\| F \|_{L^q(\Omega)} +|\lambda|^{1/2} \| f \|_{L^q(\Omega)} + \| f \|_{L^q(\Omega\cap 2B_0)} 
+ \| p \|_{L^q(\Omega\cap 2 B_0)}\\
&\qquad\qquad
+(1+ |\lambda|^{1/2} ) \| u \|_{L^q(\Omega\cap 2B_0)}
+ |\lambda| \| u \|_{W_0^{-1, q}(\Omega \cap 2B_0)}\Big\}. \\
\endaligned
$$
This, together with \eqref{E-1-1}, gives
\begin{equation}\label{E-1-2}
\aligned
 & |\lambda|^{1/2} \| \nabla u \|_{L^q(\Omega )}
+ |\lambda| \| u \|_{L^q(\Omega )}\\
& \le C \Big\{
\| F \|_{L^q(\Omega)} + ( |\lambda|^{1/2} +1) \| f\|_{L^q(\Omega)} 
+ \| p \|_{L^q(\Omega\cap2 B_0)}\\
&\qquad\qquad
+( |\lambda|^{1/2} +1 ) \| u \|_{L^q(\Omega\cap 2B_0)}
+ |\lambda| \| u \|_{W_0^{-1, q}(\Omega \cap 2B_0)}\Big\} \\
&  \le C \Big\{ 
\| F \|_{L^q(\Omega)}+ ( |\lambda|^{1/2} +1) \| f\|_{L^q(\Omega)} 
+ \|\nabla u \|_{L^q(\Omega\cap 2B_0)}\\
&\qquad\qquad
+ ( |\lambda|^{1/2} +1 ) \| u \|_{L^q(\Omega\cap 2B_0)}
+ |\lambda| \| u \|_{W_0^{-1, q}(\Omega \cap 2B_0)}\Big\}, 
\endaligned
\end{equation}
where we have  assumed $\int_{\Omega\cap 2B_0} p=0$ and used Lemma \ref{Lemma-B-2} for the last inequality.
As a result, we have proved that $u\in W^{1, q}_0(\Omega; \Cd)$. Moreover, 
we obtain  \eqref{E-1-0} if $|\lambda|\ge  \lambda_0$ and $\lambda_0>1$ is sufficiently large.
\end{proof}

\begin{remark}\label{re-E-1-0}
Suppose that $\lambda \in \Sigma_{\theta}  $ and $ |\lambda|\le \lambda_0$.
Let $2<q< \infty$.
It follows from \eqref{E-1-2}  and Theorem \ref{thm-B-2}  as well as the interior estimates for the Stokes
equations with $\lambda=0$ that 
\begin{equation}\label{E-1-3}
\aligned
|\lambda|^{1/2} \| \nabla u \|_{L^q(\Omega)} + |\lambda| \| u \|_{L^q(\Omega)}
\le C \left\{
\| F \|_{L^q(\Omega)} + \| f \|_{L^q(\Omega)}
+ \| u \|_{L^q(\Omega \cap 3B_0)} \right\},
\endaligned
\end{equation}
where $C$ depends on $d$, $q$, $\theta$ and $\Omega$.
\end{remark}

The next lemma gives the uniqueness for $q>2$.

\begin{lemma}\label{Lemma-E-2}
Let $2\le q<\infty$ and $\lambda\in \Sigma_\theta$.
Let $u \in W^{1, q}_0 (\Omega; \Cd)$ be a solution of \eqref{E-eq} in $\Omega$ with $F=0$ and $f=0$.
Then $u=0$  in $\Omega$.
\end{lemma}

\begin{proof}
The case $q=2$ is well known.
To handle the case $q>2$, we choose $\varphi \in C^\infty(\Rd; \R)$ such that $\varphi=1$ in $\Omega\setminus B(0, 2R_0)$
and $\varphi=0$ in $B(0, R_0)$, as in the proof of Lemma \ref{Lemma-E-1}.
Then the Stokes equations in \eqref{L-eq} hold in $\Rd$ with $F=0$ and $f=0$.
Since the right-hand sides of \eqref{L-eq} have compact support and thus are in $L^2(\Rd; \Cd)$, it follows from Remark \ref{re-R-u} that 
$u \varphi \in W^{1, 2} (\Rd; \Cd)$.
As a result, $u \in W^{1, 2}_0(\Omega; \Cd) $.
By the uniqueness for $q=2$, we conclude that  $u=0$ in $\Omega$.
\end{proof}

\begin{lemma}\label{Lemma-E-3}
Let $2\le  q< \infty$ and $\lambda\in \Sigma_{\theta, \delta} $.
Let $u \in W^{1, 2}_0(\Omega; \Cd) $ be an energy  solution of \eqref{E-eq}
with $F \in L^q(\Omega; \Cd)\cap L^2(\Omega; \Cd)$ and $f\in L^q(\Omega; \Cdd)\cap L^2(\Omega; \Cdd)$. Then
\begin{equation}\label{E-3-0}
|\lambda|^{1/2} \| \nabla u \|_{L^q(\Omega)}
+ |\lambda| \| u \|_{L^q(\Omega)}
\le C \left\{  \| F \|_{L^q(\Omega)} + |\lambda|^{1/2} \| f \|_{L^q(\Omega)} \right\},
\end{equation}
where $C$ depends on $d$, $q$, $\theta$, $\delta$ and $\Omega$.
\end{lemma}

\begin{proof}
The case $q=2$ is the well known energy estimate.
To handle the case $q> 2$, we argue by contradiction.
Note that by Lemma \ref{Lemma-E-1}, $u\in W^{1, q}_0(\Omega; \Cd)$.
Suppose the estimate  \eqref{E-3-0} is not true.
Then there exist sequences $\{u^\ell \}\subset  W^{1, q}_0 (\Omega; \Cd)$,
 $\{ F^\ell \}\subset L^q(\Omega; \Cd)\cap L^2(\Omega; \Cd)$,
 $\{ f^\ell \}\subset L^q(\Omega; \Cdd)\cap L^2(\Omega; \Cdd)$  and $\{\lambda^\ell\}
\subset \Sigma_{\theta, \delta}  $ such that 
\begin{equation}\label{E-3-1}
\left\{
\aligned
-\Delta u^\ell +\nabla p^\ell + \lambda^\ell  u^\ell & = F^\ell + \text{\rm div}(f^\ell)  & \quad & \text{ in } \Omega,\\
\text{\rm div} (u^\ell) & = 0& \quad & \text{ in } \Omega,
\endaligned
\right.
\end{equation}
for some $p^\ell \in L^2_{\loc} (\overline{\Omega}; \C)$, 
\begin{equation}\label{E-3-2}
|\lambda^\ell|^{1/2} \|\nabla u^\ell \|_{L^q(\Omega)} + |\lambda^\ell | \| u^\ell \|_{L^q(\Omega)}
=1,
\end{equation}
and
\begin{equation}\label{E-3-3}
\| F^\ell \|_{L^q(\Omega)}  +|\lambda^\ell |^{1/2} \| f^\ell \|_{L^q(\Omega)} \to 0 \quad \text{ as } \ell \to \infty.
\end{equation}
Since  $| \lambda^\ell|\ge \delta$, it follows from \eqref{E-3-2} that $\| u^\ell \|_{W^{1, q}_0(\Omega)} \le C$. 
By passing to a subsequence, we may assume that $u^\ell \to u$ weakly in $W^{1, q}_0(\Omega; \Cd)$.
 We may also assume that either $|\lambda^\ell |\to \infty$ or
$\lambda ^\ell \to \lambda\in \C$.

We consider three  cases: (1) $ \lambda^\ell  \to \lambda\in \C$ and $|\lambda|> 2 \lambda_0$, where $\lambda_0>1$ is given by
Lemma \ref{Lemma-E-1};  (2) $\lambda^\ell  \to \lambda$ and $|\lambda|\le 2 \lambda_0$;    and (3) $|\lambda^\ell | \to \infty$.

Case (1). Suppose $\lambda^\ell \to \lambda \in \C$ and $|\lambda|> 2\lambda_0$.
It follows that $u \in W^{1, q}_0 (\Omega; \Cd)$ is a solution of \eqref{E-eq} in $\Omega$ with
$F=0$ and $f=0$. By Lemma \ref{Lemma-E-2}, we obtain $u=0$ in $\Omega$.
Thus, $u^\ell \to 0$ weakly in $W^{1, q}_0 (\Omega; \Cd)$.
This implies that  $u^\ell \to 0$ strongly in $W_0^{-1, q} (\Omega \cap 2B_0; \Cd)$.
However, by \eqref{E-1-0}  and  \eqref{E-3-2}-\eqref{E-3-3}, we have
\begin{equation}\label{E-3-4}
\aligned
1= & |\lambda^\ell|^{1/2} \|\nabla u^\ell \|_{L^q(\Omega)}   + |\lambda^\ell | \| u^\ell \|_{L^q(\Omega)}\\
 & \le C \left\{ \| F^\ell \|_{L^q(\Omega)} +|\lambda^\ell|^{1/2} \| f^\ell \|_{L^q(\Omega)} 
+ |\lambda^\ell | \| u^\ell \|_{W_0^{-1, q}(\Omega\cap 2B_0)}\right\}
\to 0,
\endaligned
\end{equation}
which yields  a  contradiction.

Case (2). Suppose $\lambda^\ell \to \lambda$ and $|\lambda|\le 2 \lambda_0$.
As in case (1),  $u^\ell \to 0$ weakly in $W^{1, q}_0(\Omega; \Cd)$.
  It   follows from \eqref{E-1-3}  that
$$
\aligned
 & 1=|\lambda^\ell |^{1/2}  \|\nabla u^\ell \|_{L^q(\Omega)}
+ |\lambda^\ell | \| u^\ell \|_{L^q(\Omega)}\\
 & \qquad
 \le C \left\{ \| F^\ell \|_{L^q(\Omega)} + \| f^\ell \|_{L^q(\Omega)} 
+ \| u^\ell \|_{L^q(\Omega\cap 3B_0) } \right\}.
\endaligned
$$
This gives us a contradiction,  as $u^\ell \to 0$ strongly in $L^q(\Omega\cap 3B_0; \Cd)$.

Case (3). Suppose that $|\lambda^\ell | \to \infty$. In view of \eqref{E-3-2}, we have $u^\ell \to 0$ strongly in $L^q(\Omega; \Cd)$.
By passing to a subsequence, we assume that $\lambda^\ell u^\ell \to v$ weakly in $L^q(\Omega; \Cd)$.
Note that if $w\in C_0^\infty (\Omega; \Cd)$ and $ \text{\rm div}(w)=0$ in $\Omega$, then
$$
-\int_\Omega u^\ell \cdot \Delta w + \int_\Omega \lambda^\ell u^\ell \cdot w 
=\int_\Omega F^\ell \cdot w -\int_\Omega f^\ell \cdot \nabla w.
$$
By letting $\ell \to \infty$, we obtain 
$
\int_\Omega v\cdot w=0.
$ 
This implies that $v=\nabla \phi$ for some $\phi\in \hW^{1, q}  (\Omega; \C)$.
Since $\lambda^\ell u^\ell\in W^{1, q}_0(\Omega; \Cd)$ and $\text{\rm div} (\lambda^\ell u^\ell)=0$ in $\Omega$, we also have 
$\int_\Omega v \cdot \nabla \varphi=0$ for any $\varphi\in C_0^\infty(\Rd; \C)$.
It follows that $\phi \in \hW^{1, q}(\Omega;\C)$ is a solution of the Neumann problem:
$
\Delta \phi   =0 
$  in  $\Omega$ and $
\frac{\partial \phi}{\partial n}  =0 
$
  on $ \partial \Omega$.
Since $\nabla \phi \in L^q (\Omega; \Cd)$, we conclude  that $v=\nabla \phi =0$ in $\Omega$. See Lemma \ref{re-harmonic} in Appendix.
Thus, $\lambda^\ell u^\ell \to 0$ weakly in $L^q(\Omega; \Cd)$ and thus strongly in $W^{-1, q}(\Omega\cap 2B_0; \Cd)$.
Consequently, \eqref{E-3-4} holds and gives us a contradiction.
This completes the proof.
\end{proof}

\begin{proof}[\bf Proof of Theorem \ref{thm-E}]

Step 1. Assume  $2\le  q< \infty$. The uniqueness is given by Lemma \ref{Lemma-E-2}. 
Since $L^2(\Omega; \Cd)\cap L^q(\Omega; \Cd)$ is dense in $L^q(\Omega; \Cd)$,
the existence as well as the estimate \eqref{E-0-0} follows from Lemma \ref{Lemma-E-3} by a standard density argument.

Step 2. Assume $1<q<2$. As in the cases of $\Rd$ and $\Rdp$, the uniqueness follows from the existence for $q^\prime>2$, proved in Step 1.
By a duality argument, similar to that in the proof of Theorem \ref{thm-B-1}, one may show that if $F \in C_0^\infty(\Omega; \Cd)$ and
$f\in C_0^\infty (\Omega; \Cdd)$, the energy solutions of \eqref{E-eq} satisfy  the estimate \eqref{E-0-0}.
As before, the existence and the estimate \eqref{E-0-0} for $F\in L^q(\Omega; \Cd)$ and
$f\in L^q(\Omega; \Cdd)$ follow  by a density argument. 
\end{proof}

\begin{proof}[\bf Proof of Theorem \ref{main-2}]

The estimate \eqref{est-1} with $C$ depending on $\delta$ is contained in Theorem \ref{thm-E}.
To establish  the estimate \eqref{est-2} with $C$ independent of $\delta$ for $d\ge 3$, we first consider the case $q<(d/2)$ and 
argue by contradiction.
Suppose \eqref{est-2} is not true.
Then there exist sequences $\{F^\ell \}\subset L^q(\Omega; \Cd)$, $\{u^\ell \}\subset W^{1,q}_0(\Omega; \Cd)$,
$\{ \lambda^\ell \}\subset \Sigma_\theta$ such that  $\lambda^\ell \to 0$, 
\begin{equation}\label{E-4-0}
\left\{
\aligned
-\Delta u^\ell +\nabla p^\ell + \lambda^\ell u^\ell & = F^\ell, \\
\text{\rm div}(u^\ell) & =0,
\endaligned
\right.
\end{equation}
in $\Omega$, 
\begin{equation}\label{E-4-1}
 |\lambda^\ell | \| u^\ell \|_{L^q(\Omega)} =1,
 \end{equation}
and  $\| F^\ell \|_{L^q(\Omega)} \to 0$ as $\ell \to \infty$.
By Theorem \ref{thm-A-2} in the Appendix, 
$$
\|\nabla u^\ell \|_{L^s(\Omega)}
\le C \left\{ \| F^\ell \|_{L^q(\Omega)} + \|\lambda^\ell u^\ell \|_{L^q(\Omega)} \right\},
$$
where $ \frac{1}{s}=\frac{1}{q}-\frac{1}{d}$.
As a result, $\{ \nabla u^\ell \}$ is bounded in $L^s(\Omega; \Cdd)$ and by Sobolev imbedding, 
$\{ u^\ell \}$ is bounded in $L^{s_*}(\Omega; \Cd)$, where $\frac{1}{s_*}=\frac{1}{s}-\frac{1}{d}=\frac{1}{q}-\frac{2}{d}$ and
we have used the fact $u^\ell \in L^q(\Omega; \Cd)$.
By passing to a subsequence, we may assume that $\lambda^\ell u^\ell \to v$ weakly in $L^q(\Omega; \Cd)$,
$u^\ell \to u$ weakly in $L^{s_*}(\Omega; \Cd)$, and
$\nabla u^\ell \to \nabla u$ weakly in $L^s(\Omega; \Cdd)$.
Since $\lambda^\ell \to 0$, we obtain $v=0$.
It then follows from \eqref{E-4-0} that 
$-\Delta u +\nabla p=0$, $\text{\rm div} (u)=0$ in $\Omega$ and $u=0$ on $\partial \Omega$. 
Since $u \in L^{s_*} (\Omega; \Cd)$,  $\nabla u \in L^s(\Omega; \Cdd)$ and $s<d$, 
we deduce  from Lemma \ref{Lemma-U}  that $u=0$ in $\Omega$.
This implies that $u^\ell \to 0$ strongly in $L^q(\Omega\cap B; \Cd)$ for any ball $B$.
However, by \eqref{E-1-3} and \eqref{E-4-1},
we have
$$
1= |\lambda^\ell |  \| u^\ell \|_{L^q(\Omega)}
\le C \left\{ \| F^\ell \|_{L^q(\Omega)}
+ \| u^\ell \|_{L^q(\Omega\cap 3B_0)} \right\},
$$ 
which yields a contradiction.

Finally, we note that by duality, the estimate \eqref{est-2} holds for $\frac{d}{d-2}< q< \infty$.
This gives the estimate for $1<q< \infty$ in the case $d\ge 4$.
If $d=3$, the range $(3/2)\le q\le 3$ follows by using the Riesz-Thorin Interpolation Theorem.
\end{proof}

%%%%%%%%%%%%%%%%%%%%%%%%%

\section{Appendix} \label{Section-A} 

In this Appendix we prove several uniqueness and regularity  results in   exterior $C^1$ domains, which are used in the previous sections.
In the case of exterior domains with $C^2$ boundaries, the proofs may be found in  \cite{Galdi}. 

\begin{lemma}\label{re-harmonic}
Let $\Omega$ be an exterior $C^1$ domain in $\R^d$, $d\ge 2$ and $1< q< \infty$.
Suppose that $\phi\in \hW^{1, q}(\Omega; \C)$, $\Delta \phi=0$ in $\Omega$ and $n \cdot \nabla \phi=0$ on $\partial \Omega$.
Then $\phi$ is constant in $\Omega$.
\end{lemma}

\begin{proof}
By using  the  mean value property for harmonic functions and $|\nabla \phi|\in L^q(\Omega)$, we obtain 
$\nabla  \phi (x) = o(1)$ as $|x|\to \infty$.
By the expansion theorem at $\infty$  for harmonic functions \cite{Axler}, we deduce that 
$\nabla \phi (x) =O(|x|^{-1})$ for $d=2$.
In the case $d\ge 3$, we obtain $\nabla \phi(x)= O(|x|^{2-d})$.
It follows that $\phi (x) = O(\log |x|)$ for $d=3$ and  $\phi (x) = O(1)$ for $d\ge 4$.
Since $\phi$ is harmonic,  by the expansion theorem, this  implies that $\phi(x)= L  + O(|x|^{2-d})$ for some $L \in \C$ and that 
$\nabla \phi(x)=O(|x|^{1-d})$ as $|x|\to \infty$ for $d\ge 3$.
As a result, we have proved that $\nabla \phi(x)=O(|x|^{1-d})$ as $|x|\to \infty$ for $d\ge 2$.

Next,  note that since $\partial \Omega$ is $C^1$ and $n \cdot \nabla \phi=0$ on $\partial\Omega$, we have $\nabla \phi  \in L^2(\Omega\cap B(0, R); \Cd)$ for any $R>1$.
Moreover, for $R$ sufficiently large, 
$$
\aligned
\int_{\Omega\cap B(0, R)} |\nabla \phi |^2
 & =\int_{\partial B(0, R)} \frac{\partial \phi }{\partial n} (\phi-\beta)\\
 & \le \| \nabla \phi  \|_{L^2(\partial B(0, R))} \| \phi  -\beta \|_{L^2(\partial B(0, R))}
\le C R \|\nabla \phi \|^2_{L^2(\partial B(0, R))},
\endaligned
$$
where  $\beta =\fint_{\partial B(0, R)} \phi $ and we have used a Poincar\'e inequality on $\partial B(0, R)$.
By letting $R\to \infty$ and using $\nabla \phi(x)=O(|x|^{1-d})$ as $|x|\to \infty$ for $d\ge 2$,
 we see that $ \| \nabla \phi \|_{L^2(\Omega)} =0$ if $d\ge 3$ and 
$\|\nabla \phi \|_{L^2(\Omega)} < \infty$ if $d=2$.
As a result,  $\nabla \phi=0$ and $\phi$ is constant  in $\Omega$ for $d\ge 3$.
Finally, to handle the case $d=2$, we use the Caccioppoli  inequality,
\begin{equation}
\aligned
\int_{\Omega\cap B(0, R)} |\nabla \phi |^2
 & \le \frac{C}{R^2} \int_{B(0, 2R)\setminus B(0, R)}  |\phi - \alpha |^2\\
 & \le C_0 \int_{B(0, 2R)\setminus B(0, R)} |\nabla \phi|^2,
\endaligned
\end{equation}
for $R$ large, 
where $\alpha =\fint_{B(0, 2R)\setminus B(0, R)} \phi$ and we have used a Poincar\'e inequality.
It follows that
$$
\int_{\Omega\cap B(0, R)} |\nabla \phi|^2 \le \frac{C_0}{C_0 +1} \int_{\Omega\cap B(0, 2R)} |\nabla \phi|^2.
$$
By letting $R \to \infty$, we obtain $\|\nabla \phi \|_{L^2(\Omega)} \le  c_0 \|\nabla \phi \|_{L^2(\Omega)}$ for some $c_0<1$.
This implies that $\|\nabla \phi \|_{L^2(\Omega)}=0$ if $\|\nabla \phi\|_{L^2(\Omega)}< \infty$.
Consequently, we conclude that $\nabla \phi=0$ and $\phi$ is constant  in  $\Omega$  for $d\ge 2$.
\end{proof}

\begin{lemma}\label{Lemma-U}
Let $\Omega$ be an exterior $C^1$ domain in $\R^d$, $d\ge 2$. Let  $1<q<  d $ and $\frac{1}{q_*}=\frac{1}{q}-\frac{1}{d}$.
Suppose that $u \in L^{q_*} (\Omega; \Cd)$, $\nabla u \in L^q(\Omega; \Cdd)$, $u=0$ on $\partial\Omega$, and
\begin{equation}\label{A-20}
-\Delta u +\nabla p  =0 \quad \text{ and} \quad
\text{\rm div}(u)  =0
\end{equation}
hold in $\Omega$ in the sense of distributions. Then $u=0$ in $\Omega$.
\end{lemma}

\begin{proof} The proof is similar to that of Lemma \ref{re-harmonic} for the case $d\ge 3$.
By the interior estimates for the Stokes equations,
\begin{equation}\label{A-22}
|x| |\nabla^2 u (x)| + |\nabla u(x)|
\le C\left(\fint_{B(x, R/4)} |\nabla u|^q \right)^{1/q},
\end{equation}
where $R=|x|$ is sufficiently large. It   follows from  $|\nabla u|\in L^q(\Omega)$ that 
$ \nabla u (x) = o(|x|^{-\gamma}) $ as $|x|\to \infty$, where $\gamma = (d/q)$.
Since $\gamma>1$, this implies that $\lim_{|x|\to \infty}  u(x)$ exists.
Using $u \in L^{q_*} (\Omega; \Cd)$, we deduce that $u(x)=o(1)$ as $|x|\to \infty$.
Also note that by the interior estimates,
$\nabla^2 u (x)= o(|x|^{-\gamma -1}) $ as $|x|\to \infty$.
Thus, $\nabla p(x) =o(|x|^{-\gamma-1}) $.
It follows that $\lim_{|x|\to \infty} p(x)$ exists.
By subtracting a constant, we may assume that $\lim_{|x|\to \infty} p(x)=0$.
As a result, we obtain $p(x)= o(|x|^{-\gamma}) $ as $|x|\to \infty$.

Next, assume $d\ge 3$. We use the Green representation formula for the Stokes equations in the domain $D_R = \{ x:  R_0< |x|< R\}$ to write $(u(x), p(x)) $ as 
a sum of layer potentials on $\partial D_R=\partial B(0, R) \cup \partial B(0, R_0)$.
Since $ |\nabla u(x)| + |p(x)| =o(|x|^{-\gamma})$, where $\gamma>1$,  and $|u(x)| = o(1)$ as $|x|\to \infty$, it is not hard to see that  the 
layer potentials on $\partial B(0, R)$  converge to $0$ as $R \to \infty$.
This allows to upgrade the decay of $(u, p)$ at $\infty$ to 
\begin{equation}\label{A-21}
|x|^{-1} | u(x)| + |\nabla u(x) | + |p(x)| = O(|x|^{1-d}) \quad \text{ as } |x|\to \infty
\end{equation}
for $d\ge 3$.

Finally, we note that since $\partial \Omega$ is $C^1$ and $u=0$ on $ \partial \Omega$, we have
$u \in W^{1, 2} (\Omega \cap B(0, R); \C^d)$ for any $R>1$.
Moreover, for $R>1$ large, 
$$
\int_{\Omega\cap B(0, R)} |\nabla u|^2=
\int_{\partial B(0, R)} \left(\frac{\partial u}{\partial n} - n p \right) \cdot u. 
$$
In view of \eqref{A-21} for $d\ge 3$ as well as the decay estimates,  $ u(x)=o(1)$ and
$|\nabla u (x)| +|p(x)| =o(|x|^{-\gamma} ) $  for $d=2$, by  letting $R\to \infty$, 
we obtain $ \|\nabla u \|_{L^2(\Omega)} =0$.
Since $u=0$ on $\partial \Omega$, it  follows that $u=0$ in $\Omega$.
\end{proof}

The following theorem is used in the proof of  the estimate \eqref{est-3} for small $|\lambda|$.

\begin{thm}\label{thm-A-2}
Let $\Omega$ be an exterior $C^1$ domain in $\R^d$, $d\ge 3$ and  $1<q<  (d/2) $.
Let $u\in W^{1, q}_0(\Omega; \Cd)$ be 
a solution of
\begin{equation}\label{A-10}
-\Delta u +\nabla p    =F \quad \text{ and } \quad 
\text{\rm div}(u)  =0
\end{equation}
in $\Omega$, where $F\in L^q(\Omega; \Cd)$. Then $u\in W^{1, s}_0(\Omega; \Cd)$ and
\begin{equation}\label{A-11}
\| \nabla u \|_{L^s(\Omega)} 
\le C \| F \|_{L^q(\Omega)},
\end{equation}
where $\frac{1}{s} =\frac{1}{q}-\frac{1}{d}$ and $C$ depends on $d$, $q$ and $\Omega$.
\end{thm}

\begin{proof}

Since $W^{1, q}_0(\Omega; \Cd)\subset L^s(\Omega; \Cd) $. 
It suffices to prove \eqref{A-11}. We divide the proof into two steps.

Step 1. We show that the solution $u$ satisfies the estimate,
\begin{equation}\label{A-12}
\|\nabla u \|_{L^s(\Omega)}
\le C \left\{ \| F \|_{L^q(\Omega)} + \| u \|_{L^q(\Omega\cap B_0)} \right\},
\end{equation}
where $B_0= B(0, 2R_0)$ and $R_0>1$ is sufficiently large. 
To this end, we choose $R_0>1$ such that $\Omega\setminus B(0, R_0)  = \Rd\setminus B(0, R_0)$ and
$\Omega \cap B(0, 2R_0)$ is a bounded $C^1$ domain.
Choose $\varphi_1 \in C_0^\infty(\Rd; \R)$ such that $\varphi_1=1$ in $\Omega\setminus B(0, (3/2)R_0)$ and
$\varphi_1 =0 $ in $B(0, (5/4) R_0)$.
Let $\varphi_2=1-\varphi_1$.
Then 
$$
\left\{
\aligned
-\Delta (u\varphi_1) +\nabla (p\varphi_1)
& =F\varphi_1 -2 (\nabla u) (\nabla \varphi_1)  - u \Delta \varphi_1 + p \nabla \varphi_1,\\
\text{\rm div}(u\varphi_1)
&=u\cdot  \nabla \varphi_1
\endaligned
\right.
$$
in $\Rd$.
It follows  from the $W^{2, q}$ estimates \cite{Galdi} for the Stokes equations (with $\lambda=0$) in $\Rd$ that 
\begin{equation}\label{A-13}
\aligned
\|\nabla (u \varphi_1)\|_{L^s(\Rd)}
& \le C \Big\{ \| F\varphi_1 \|_{L^q(\Rd)}
+ \| (\nabla u) (\nabla \varphi_1)\|_{L^q(\Rd)}\\
&\qquad\qquad
+ \| u \Delta \varphi_1\|_{L^q(\Rd)}
+ \| p \nabla \varphi_1 \|_{L^q(\Rd)} + \| \nabla (u \nabla \varphi_1) \|_{L^q(\Rd)}
 \Big\}.
\endaligned
\end{equation}
Let $\Omega_0 = \Omega \cap B(0, 2R_0)$.
Note that $u \varphi_2=0$ on $\partial \Omega_0=\partial\Omega \cup \partial B(0, 2R_0)$ and 
$$
\left\{
\aligned
-\Delta (u\varphi_2) +\nabla (p\varphi_2)
& =F\varphi_2 -2 (\nabla u) (\nabla \varphi_2)  - u \Delta \varphi_2 + p \nabla \varphi_2,\\
\text{\rm div}(u\varphi_2)
&=u\cdot  \nabla \varphi_2
\endaligned
\right.
$$
in $\Omega_0$.
It follows from the $W^{1, q}$ estimates for the Stokes equations (with $\lambda=0$) in the  $C^1$ domain $\Omega_0$  that
\begin{equation}\label{A-14} 
\aligned
\|\nabla (u\varphi_2)\|_{L^s(\Omega_0)}
 & \le C \Big\{ 
\| F \varphi_2 \|_{L^q(\Omega_0)}
+ \| (\nabla u) (\nabla \varphi_2) \|_{L^q(\Omega_0)}
+ \| u \Delta \varphi_2\|_{L^q(\Omega_0)}\\
 &\qquad\qquad  + \| p \nabla \varphi_2\|_{L^q (\Omega_0)} 
+\| u \nabla \varphi_2 \|_{L^q(\Omega_0)} \Big\}.
\endaligned
\end{equation}
See Remark \ref{re-B-2}.
The estimate \eqref{A-12} follows from \eqref{A-13} and \eqref{A-14} as well as the interior estimates for
the Stokes equations.

\medskip

Step 2. We establish the estimate \eqref{A-11} by a compactness argument.

Suppose \eqref{A-11} is not true.
Then there exist sequences $\{F^\ell \}\subset L^q(\Omega; \Cd)$, $\{ u^\ell \}\subset W^{1, q}_0 (\Omega; \Cd)\cap W^{1, s}_0(\Omega; \Cd)$, such that 
\begin{equation}\label{A-15}
\left\{
\aligned
-\Delta u^\ell  +\nabla p^\ell  & =F^\ell,\\
\text{\rm div}(u^\ell ) & =0
\endaligned
\right.
\end{equation}
hold in $\Omega$ for some $p^\ell \in L^1_{\loc} (\Omega; \C)$,
\begin{equation}\label{A-16}
\|\nabla u^\ell \|_{L^s(\Omega)}  =1,
\end{equation}
and $\| F^\ell \|_{L^q(\Omega)} \to 0$, as $\ell \to \infty$.
Since $\|u^\ell \|_{L^{s_*}(\Omega)}\le C \| \nabla u^\ell  \|_{L^s(\Omega)} =  C$, where $\frac{1}{s_*}=\frac{1}{s}-\frac{1}{d}$, 
by passing to a subsequence, we may assume $u^\ell \to u$ weakly in $L^{s_*}(\Omega; \Cd)$ and
$\nabla u^\ell \to \nabla u $ weakly in $L^s(\Omega; \Cdd)$.
It follows that $u$ is a solution of \eqref{A-10} with $F=0$.
Note that $q<(d/2)$ implies $s=q_*<d$. 
Thus,  by Lemma \ref{Lemma-U},   $u=0$ in $\Omega$.
This implies that $u^\ell \to 0$ strongly in $L^q(\Omega\cap B_0; \Cd)$.
However, by \eqref{A-12},
$$
\|\nabla u^\ell \|_{L^s(\Omega)}
\le C \left\{ \| F^\ell \|_{L^q(\Omega)}
+ \| u^\ell \|_{L^q(\Omega\cap B_0)} \right\},
$$
which leads to a contradiction with \eqref{A-16} if we let $\ell \to \infty$.
\end{proof}

Recall that 
$
C_{0, \sigma}^\infty (\Omega)
=\left\{ u\in C_0^\infty (\Omega; \Cd): \text{\rm div}(u)=0 \text{ in } \Omega \right\}.
$
Let  $L^q_\sigma (\Omega)$ denote  the closure of $C_{0, \sigma}^\infty(\Omega)$ in $L^q(\Omega; \Cd)$ and
$$
G_q (\Omega) =\left \{ u: u=\nabla p  \text{ for some } p \in \hW^{1, q} (\Omega; \C) \right\}.
$$

\begin{thm}\label{thm-A-1}
Let $\Omega$ be a bounded or exterior domain with $C^1$ boundary in $\Rd$, $d\ge 2$. Then
\begin{equation}\label{H-de}
L^q(\Omega; \Cd)
= L^q_\sigma (\Omega) \oplus G_q (\Omega)
\end{equation}
for $1< q< \infty$. That is, for any $u\in L^q(\Omega; \Cd)$, there exists a unique $(v, w) \in L^q_\sigma (\Omega)
\times G_q(\Omega)$ such that $u=v+w$ in $\Omega$ and
\begin{equation}
\| v\|_{L^q(\Omega)} + \| w \|_{L^q(\Omega)}\le C \| u \|_{L^q(\Omega)},
\end{equation}
where $C$ depends on $d$, $q$ and $\Omega$.
\end{thm}

The formula \eqref{H-de} is referred to as the Helmholtz decomposition,
which is well known in  the case of bounded or exterior domains with smooth boundaries (see \cite{Sohr-1994} for references).
In the case of bounded or exterior domains with $C^1$ boundaries,  a sketch of the proof for \eqref{H-de} 
may be found in \cite{Sohr-1994}.  Also see \cite{Fabes-1998}.
The decomposition also holds for $1<q<\infty$ if $\Omega$ is a bounded convex domain \cite{Geng-2010}.
If $\Omega$ is a bounded or exterior domain with Lipschitz boundaries,
the Helmholtz decomposition \eqref{H-de} holds if 
\begin{equation}\label{lip-r}
\left\{
\aligned
(3/2)-\e <  & q< 3+\e  & \quad & \text{ for } d\ge 3,\\
(4/3)-\e < & q < 4 +\e & \quad & \text{ for  } d=2,
\endaligned
\right.
\end{equation}
where $\e>0$ depends on $\Omega$. The ranges in \eqref{lip-r} are known to be sharp. See \cite{Fabes-1998}.
We remark that Theorem \ref{thm-A-1} is not used in this paper. 

\medskip

\noindent{\bf Conflict of interest.}  The authors declare that there is no conflict of interest.

\medskip

\noindent{\bf Data availability. }
 Data sharing not applicable to this article as no datasets were generated or analyzed during the current study.

 \bibliographystyle{amsplain}
 
\bibliography{Geng-Shen2023.bbl}

\providecommand{\bysame}{\leavevmode\hbox to3em{\hrulefill}\thinspace}
\providecommand{\MR}{\relax\ifhmode\unskip\space\fi MR }
\providecommand{\MRhref}[2]{%
  \href{http://www.ams.org/mathscinet-getitem?mr=#1}{#2}
}
\providecommand{\href}[2]{#2}
\begin{thebibliography}{10}

\bibitem{Axler}
Sheldon Axler, Paul Bourdon, and Wade Ramey, \emph{Harmonic function theory},
  second ed., Graduate Texts in Mathematics, vol. 137, Springer-Verlag, New
  York, 2001.

\bibitem{Sohr-1987}
Wolfgang Borchers and Hermann Sohr, \emph{On the semigroup of the {S}tokes
  operator for exterior domains in {$L^q$}-spaces}, Math. Z. \textbf{196}
  (1987), no.~3, 415--425.

\bibitem{Borchers-two}
Wolfgang Borchers and Werner Varnhorn, \emph{On the boundedness of the {S}tokes
  semigroup in two-dimensional exterior domains}, Math. Z. \textbf{213} (1993),
  no.~2, 275--299.

\bibitem{Breit-2022}
Dominic Breit, \emph{Partial boundary regularity for the {Navier-Stokes}
  equations in irregular domains}, arXiv:2208.00415v2 (2022).

\bibitem{Brown-S-1995}
Russell~M. Brown and Zhongwei Shen, \emph{Estimates for the {S}tokes operator
  in {L}ipschitz domains}, Indiana Univ. Math. J. \textbf{44} (1995), no.~4,
  1183--1206.

\bibitem{Deuring-2001}
Paul Deuring, \emph{The {S}tokes resolvent in 3{D} domains with conical
  boundary points: nonregularity in {$L^p$}-spaces}, Adv. Differential
  Equations \textbf{6} (2001), no.~2, 175--228.

\bibitem{Mitrea-2004}
Martin Dindo\v{s} and Marius Mitrea, \emph{The stationary {N}avier-{S}tokes
  system in nonsmooth manifolds: the {P}oisson problem in {L}ipschitz and
  {$C^1$} domains}, Arch. Ration. Mech. Anal. \textbf{174} (2004), no.~1,
  1--47.

\bibitem{FKV-1988}
E.~B. Fabes, C.~E. Kenig, and G.~C. Verchota, \emph{The {D}irichlet problem for
  the {S}tokes system on {L}ipschitz domains}, Duke Math. J. \textbf{57}
  (1988), no.~3, 769--793.

\bibitem{Fabes-1998}
Eugene Fabes, Osvaldo Mendez, and Marius Mitrea, \emph{Boundary layers on
  {S}obolev-{B}esov spaces and {P}oisson's equation for the {L}aplacian in
  {L}ipschitz domains}, J. Funct. Anal. \textbf{159} (1998), no.~2, 323--368.

\bibitem{Sohr-1994}
Reinhard Farwig and Hermann Sohr, \emph{Generalized resolvent estimates for the
  {S}tokes system in bounded and unbounded domains}, J. Math. Soc. Japan
  \textbf{46} (1994), no.~4, 607--643.

\bibitem{Kato-1964}
Hiroshi Fujita and Tosio Kato, \emph{On the {N}avier-{S}tokes initial value
  problem. {I}}, Arch. Rational Mech. Anal. \textbf{16} (1964), 269--315.

\bibitem{Tolksdorf-2022}
Fabian Gabel and Patrick Tolksdorf, \emph{The {S}tokes operator in
  two-dimensional bounded {L}ipschitz domains}, J. Differential Equations
  \textbf{340} (2022), 227--272.

\bibitem{Galdi}
G.~P. Galdi, \emph{An introduction to the mathematical theory of the
  {N}avier-{S}tokes equations}, second ed., Springer Monographs in Mathematics,
  Springer, New York, 2011, Steady-state problems.

\bibitem{Kilty-2015}
Jun Geng and Joel Kilty, \emph{The {$L^p$} regularity problem for the {S}tokes
  system on {L}ipschitz domains}, J. Differential Equations \textbf{259}
  (2015), no.~4, 1275--1296.

\bibitem{Geng-2010}
Jun Geng and Zhongwei Shen, \emph{The {N}eumann problem and {H}elmholtz
  decomposition in convex domains}, J. Funct. Anal. \textbf{259} (2010), no.~8,
  2147--2164.

\bibitem{Giga-1981}
Yoshikazu Giga, \emph{Analyticity of the semigroup generated by the {S}tokes
  operator in {$L_{r}$} spaces}, Math. Z. \textbf{178} (1981), no.~3, 297--329.

\bibitem{Mitrea-2008}
Marius Mitrea and Sylvie Monniaux, \emph{The regularity of the {S}tokes
  operator and the {F}ujita-{K}ato approach to the {N}avier-{S}tokes initial
  value problem in {L}ipschitz domains}, J. Funct. Anal. \textbf{254} (2008),
  no.~6, 1522--1574.

\bibitem{Mitrea-2009}
\bysame, \emph{On the analyticity of the semigroup generated by the {S}tokes
  operator with {N}eumann-type boundary conditions on {L}ipschitz subdomains of
  {R}iemannian manifolds}, Trans. Amer. Math. Soc. \textbf{361} (2009), no.~6,
  3125--3157.

\bibitem{Mitrea-2012}
Marius Mitrea and Matthew Wright, \emph{Boundary value problems for the
  {S}tokes system in arbitrary {L}ipschitz domains}, Ast\'{e}risque (2012),
  no.~344, viii+241.

\bibitem{Shen-2012}
Zhongwei Shen, \emph{Resolvent estimates in {$L^p$} for the {S}tokes operator
  in {L}ipschitz domains}, Arch. Ration. Mech. Anal. \textbf{205} (2012),
  no.~2, 395--424.

\bibitem{Solo-1977}
V.A. Solonnokov, \emph{Estimate for solutions of nonstationary {Navier-Stokes}
  equations}, J. Sov. Math. \textbf{8} (1977), 467--529.

\bibitem{Stein}
Elias~M. Stein, \emph{Singular integrals and differentiability properties of
  functions}, Princeton Mathematical Series, No. 30, Princeton University
  Press, Princeton, NJ, 1970.

\bibitem{Taylor-2000}
Michael~E. Taylor, \emph{Incompressible fluid flows on rough domains},
  Semigroups of operators: theory and applications ({N}ewport {B}each, {CA},
  1998), Progr. Nonlinear Differential Equations Appl., vol.~42,
  Birkh\"{a}user, Basel, 2000, pp.~320--334.

\bibitem{Tolksdorf-2018}
Patrick Tolksdorf, \emph{On the {${\rm L}^p$}-theory of the {N}avier-{S}tokes
  equations on three-dimensional bounded {L}ipschitz domains}, Math. Ann.
  \textbf{371} (2018), no.~1-2, 445--460.

\bibitem{Tolksdorf-2020}
\bysame, \emph{The {S}tokes resolvent problem: optimal pressure estimates and
  remarks on resolvent estimates in convex domains}, Calc. Var. Partial
  Differential Equations \textbf{59} (2020), no.~5, Paper No. 154, 40.

\end{thebibliography}

\bigskip

\begin{flushleft}

Jun Geng, School of Mathematics and Statistics, Lanzhou University, Lanzhou, People's Republic of China.

E-mail:  gengjun@lzu.edu.cn

\bigskip

Zhongwei Shen,
Department of Mathematics,
University of Kentucky,
Lexington, Kentucky 40506,
USA.

E-mail: zshen2@uky.edu
\end{flushleft}

\bigskip

\medskip

\end{document}